\documentclass{siamltex}
\usepackage{amssymb}
\usepackage{psfrag}
\usepackage{subfigure}
\usepackage{graphicx}
\usepackage{color}
\usepackage{algorithm,algorithmicx,algpseudocode}
\usepackage{multirow,rotating}

\newtheorem{remark}{Remark}[section]

\def\IR{\hbox{\rm I\kern-.2em\hbox{\rm R}}}

\def\as{{\cal A}_{seed}}
\def\ak{{\cal A}_k}
\def\a{{\cal A}}
\def\pex{{\cal P}_{ex}}
\def\pin{{\cal P}_{inex}}
\def\ps{{\cal P}_{seed}}
\def\pup{{\cal P}_{upd}}
\def\pk1{{\cal P}_{in,F}}

\def\ss{S_{seed}}
\def\sin{S_{inex}}
\def\sup{S_{upd}}

\def\tunos{\Theta_{seed}^{(1)}}
\def\tdues{\Theta_{seed}^{(2)}}
\def\tunok{\Theta^{(1)}}
\def\tduek{\Theta^{(2)}}

\begin{document}

\title {Updating constraint preconditioners  \\
for KKT systems in quadratic programming \\
via low-rank corrections\thanks{Work partially
supported by INdAM-GNCS, under the 2013 Projects \emph{Strategie risolutive per
sistemi lineari di tipo KKT con uso di informazioni strutturali}
and \emph{Metodi numerici e software per l'ottimizzazione
su larga scala con applicazioni all'image processing}.}
}
\author {Stefania Bellavia\thanks{Dipartimento  di Ingegneria Industriale, Universit\`a degli Studi di Firenze,
viale Morgagni 40,  50134 Firenze,  Italy,
stefania.bellavia@unifi.it, benedetta.morini@unifi.it.}
\and Valentina De Simone\thanks{Dipartimento di Matematica e Fisica, Seconda Universit\`a degli Studi di Napoli,  
viale A. Lincoln~5, 81100 Caserta, Italy, valentina.desimone@unina2.it, daniela.diserafino@unina2.it.}
\and Daniela di Serafino$^\ddag$\thanks{Istituto di Calcolo e Reti ad Alte Prestazioni,
CNR, via P. Castellino 111, 80131 Napoli, Italy.}
\and Benedetta Morini$^\dag$
}


\newpage

\maketitle
\vskip 5 pt 
\noindent
\centerline{\small Revised version -- July 7, 2015}
\vskip 8 pt

\begin{abstract}
This work focuses on the iterative solution of sequences of KKT linear systems arising in interior
point methods applied to large convex quadratic programming problems. This task is
the computational core of the interior point procedure and an efficient preconditioning
strategy is crucial for the efficiency of the overall method. Constraint preconditioners
are very effective in this context; nevertheless, their computation may be very expensive
for large-scale problems and resorting to approximations of them may be convenient.
Here we propose a procedure for building inexact constraint preconditioners by
updating a {\em seed} constraint preconditioner computed for a KKT matrix at a previous
interior point iteration. These updates are obtained through low-rank corrections of the Schur
complement of the (1,1) block of the seed preconditioner. The updated preconditioners
are analyzed both theoretically and computationally.
The results obtained show that our updating procedure, coupled with 
an adaptive strategy for determining whether to reinitialize or update the preconditioner,
can enhance the performance of interior point methods on large problems.
\end{abstract}

\maketitle

\section{Introduction\label{sec:intro}}
Second-order methods for constrained and unconstrained optimization require, at each iteration, 
the solution of a  system of linear equations. For large-scale problems it is common to solve 
each  system by an iterative method coupled with a suitable preconditioner.
Since the computation of  a preconditoner for  each linear system can be very expensive, in recent years
there has been a growing interest  in reducing the cost for preconditioning
by sharing some computational effort through subsequent linear systems.

Updating preconditioner frameworks for sequences of linear systems have a common feature: based on information
generated during the solution of a linear system in the sequence, a preconditioner for a subsequent system is generated.
An efficient updating procedure is expected to build a preconditioner which is
less effective in terms of linear iterations than the one computed from scratch, but more convenient
in terms of cost for the overall linear algebra phase. 
The approaches existing in literature can be broadly classified as: 
limited-memory quasi-Newton preconditioners for 
symmetric positive definite and nonsymmetric matrices (see, e.g., \cite{bbmp, gst, mn}),
recycled Krylov information preconditioners for symmetric and nonsymmetric matrices
(see, e.g., \cite{cdg, fr, ggm, lrt}),
updates of factorized preconditioners for symmetric positive definite  and nonsymmetric matrices
(see, e.g., \cite{bbm, bddm_sisc, bddm_sinum, bmp, bb, ccs, tt1,m}).

In this paper we study the problem of preconditioning sequences of KKT systems arising in the
solution of the convex quadratic programing (QP) problem 
\begin{equation}\label{CQP}
\begin{array}{ll}
\mbox{minimize} & \displaystyle \frac{1}{2} x^T Q x + c^Tx , \\[2mm]
\mbox{subject to}           & A_1 x - s = b_1, \;\; A_2 x = b_2, \;\; x + v = u, \;\; (x,s,v) \ge 0
\end{array}
\end{equation}
by Interior Point (IP) methods \cite{g,w}. Here $Q\in \mathbb{R}^{n\times n}$ is symmetric positive semidefinite,
and $A_1\in \mathbb{R}^{m_1\times n}$, $A_2\in \mathbb{R}^{m_2\times n}$, with $m=m_1+m_2\le n$.
Note that $s$ and $v$ are slack variables, used to transform the inequality constraints $A_1 x \ge b_1$ and $x \le u$
into equality constraints.

The updating strategy proposed here concerns constraint preconditioners (CPs)
\cite{bgz, cddd_coap1, ddd, d, fgg, kgw, ps}. Using the factorization of a {\em seed} CP 
computed for some KKT matrix of the sequence, the factorization of an approximate CP is built 
for subsequent systems. The resulting preconditioner is a special case of
{\em inexact} CP \cite{bgvz,dr,lv,ps,ss} and is intended to reduce the computational
cost while preserving effectiveness. To the best of our knowledge, this is the first attempt to 
study, both theoretically and numerically, a preconditioner updating technique
for sequences of KKT systems arising
from (\ref{CQP}). A computational analysis of the reuse of CPs in consecutive IP steps,
which can be regarded as a limit case of preconditioner updating, has been presented in \cite{cddd_simai}.

Concerning preconditioner updates for sequences of systems arising from IP methods, 
we are aware of  the work in \cite{bsz, wol}, where linear programming  problems are considered.
In this case, the KKT linear systems are reduced to the so-called normal equation form; some systems
of the sequence are solved by the Cholesky factorization, while the remaining ones are solved by
the Conjugate Gradient method preconditioned  with  a low-rank correction  of the last computed Cholesky factor. 
 
Motivated  by \cite{bsz, wol}, in this work we adapt the low-rank corrections given therein to our problem.
In our approach the updated preconditioner is an inexact CP
where the Schur complement of the (1,1) block is replaced by a low-rank 
modification of the corresponding Schur complement in the seed preconditioner.
The validity of the proposed procedure is supported by a spectral analysis of the preconditioned matrix
and by numerical results illustrating its performance.

The paper is organized as follows. In Section~\ref{sec2} we provide preliminaries on CPs 
and new spectral analysis results for general inexact CPs. In Section~\ref{sec3} we present our 
updating procedure and specialize the spectral analysis conducted in the
previous section to the updated preconditioners. In Section~\ref{sec:experiments} we discuss implementation
issues of the updating procedure and present numerical results obtained by solving sequences of linear
systems arising in the solution of convex QP problems.
These results show that our updating technique is able to reduce the computational cost for solving
the overall sequence whenever the updating strategy is performed in conjunction with 
adaptive strategies for determining whether to recompute the seed preconditioner or update
the current one.

In the following, $\| \cdot \|$ denotes the vector or matrix 2-norm
and, for any symmetric matrix $A$, $\lambda_{\min}(A)$ and $\lambda_{\max}(A)$ denote
its minimum and maximum eigenvalues. Finally, for any complex number
$\alpha$, $\mathcal{R}(\alpha)$ and $\mathcal{I}(\alpha)$ denote
the real and imaginary parts of $\alpha$, respectively.

\section{Inexact Constraints Preconditioners\label{sec2}}

In this section we discuss the features of the KKT matrices arising in the solution of problem 
(\ref{CQP}) by IP methods and present spectral properties of both CPs and
inexact CPs. This analysis  will be used to develop our updating strategy.
Throughout the paper we assume that the matrix 
$A = \left[ \begin{array}{c} A_1 \\[1mm]  A_2  \end{array} \right]\in \mathbb{R}^{m\times n}$ is full rank. 

The application of an IP method to problem (\ref{CQP}) gives rise to a sequence of 
symmetric indefinite matrices differing by a diagonal, possibly indefinite, matrix.
In fact, at the $k$th IP iteration, the KKT matrix takes the form
$$
\ak=
\left[
\begin{array}{cc}
Q+\tunok_k  & A^T  \\
A   &     -\tduek_k   
\end{array}
\right],
$$
where $\Theta_k^{(1)} \in \mathbb{R}^{n \times n}$ is diagonal positive definite and 
$\Theta_k^{(2)} \in \mathbb{R}^{m \times m}$ is diagonal positive semidefinite. 
In particular, 
$$
\begin{array}{l}
\Theta_k^{(1)} = X_k^{-1}W_k + V_k^{-1} T_k, \\[1mm]
\Theta_k^{(2)} = \left[ 
\begin{array}{cc}
Y_k^{-1}S_k & 0  \\
0   &   0    
\end{array}
\right],
\end{array}
$$
where $(x_k, w_k), \;  (s_k,y_k)$, and $( v_k, t_k)$ are 
the pairs of complementary variables of problem (\ref{CQP}) evaluated at the current iteration,
and $X_k,W_k,S_k,Y_k,V_k$ and $T_k$ are the corresponding diagonal matrices according to the standard 
IP notation. If the QP problem has no linear inequality constraints then $\Theta_k^{(2)}$ is the zero matrix;
otherwise $\Theta_k^{(2)}$  admits positive diagonal entries corresponding to 
slack variables for linear inequality constraints. 

To simplify the notation,  in the rest of the paper we drop the iteration index $k$ from $\ak$,
$\Theta_k^{(1)}$ and $\Theta_k^{(2)}$. Hence, 
\begin{equation}\label{a}
\a=\ak
\end{equation}
and the CP for $\a$ is given by
\begin{equation}\label{pg}
\pex=\left[ 
\begin{array}{cc}
G  & A^T  \\
A   &   -\tduek    
\end{array}
\right],
\end{equation}
where $G$ is an approximation to $Q+\tunok$. Clearly, the application of $\pex$ requires its factorization.

In order to define $\pex$, one possibility is to specify $G$ and then explicitly factorize
the preconditioner; alternatively, implicit factorizations can be used, where $\pex$ is derived from
specially chosen factors $M$ and $C$ in $\pex = MCM^T$ (see, e.g.,  \cite{dgsw,dw}). 
We follow the former approach and choose $G$ as the diagonal matrix with the same diagonal
entries as $Q+\tunok$, i.e.,
\begin{equation}\label{g}
G={\rm diag}(Q+\tunok).
\end{equation}
The resulting matrix $\pex^{-1}\a$  
has an eigenvalue at 1 with multiplicity $2m-p$,  with $p \! = \! rank(\tduek)$, and 
$n-m+p$ real positive eigenvalues such that the  better $G$ approximates $Q+ \tunok$ the more clustered around 1
they are \cite{d, kgw}.

The factorization of $\pex$ can be obtained by computing 
a Cholesky-like factorization of  the negative Schur complement, $S$, 
of $G$ in $\a$,
\begin{equation}\label{S}
S=AG^{-1}A^T +\tduek,
\end{equation}
and using the  block decomposition
\begin{equation}\label{pexfatt}
\pex=
\left[ 
\begin{array}{cc}
I_n  & 0  \\
AG^{-1}  & I_m      
\end{array}
\right]
\left[ 
\begin{array}{cc}
G  &  0  \\
0   &    -S    
\end{array}
\right]\left[ 
\begin{array}{cc}
I_n  & G^{-1}A^T  \\
0   & I_m    
\end{array}
\right],
\end{equation}
where $I_r$ is the identity matrix of dimension $r$, for any integer $r$.

In problems where a large part of the computational cost for  
solving the linear system depends on the computation of a Cholesky-like factorization of $S$,
this matrix may be replaced by a computationally cheaper approximation of it \cite{dr, lv,ps}.
This yields the following {\em inexact} CP:
$$
\pin=\left[ 
\begin{array}{cc}
I_n  & 0  \\
AG^{-1}  & I_m     
\end{array}
\right]
\left[ 
\begin{array}{cc}
G  &  0  \\
0   &    -\sin
\end{array}
\right]\left[ 
\begin{array}{cc}
I_n  & G^{-1}A^T  \\
0   & I_m       
\end{array}
\right] ,
$$
where $\sin$ is the approximation of $S$.

\subsection{Analysis of the eigenvalues}

Due to the indefiniteness of $\pin^{-1}\a$ 
and the specific form of $\pin$, the preconditioned linear system must be solved by either 
nonsymmetric solvers, such as GMRES \cite{gmres} and QMR \cite{fn2}, or by the simplified version  
of the QMR method \cite{fn}; such a simplified solver, named SQMR, is capable of exploiting
the symmetry of $\a$ and $\pin$ and is used in our numerical experiments.

Although the convergence of many Krylov solvers 
is not fully characterized by the spectral properties of the coefficient matrix, in many practical cases 
it depends to a large extent on the eigenvalue distribution.
For this reason, exploiting the spectral analysis of $\pin^{-1}\a$ made in the general context of
saddle point problems \cite{bs,b,ss},  we provide
bounds on the eigenvalues which will guide our preconditioner updates.

In order to carry out our analysis we write the eigenvalue problem for the matrix
$\pin^{-1}\a$ in the form
\begin{equation}\label{eig_pb}
\left[
\begin{array}{cc}
Q+\tunok  & A^T  \\
A   &   -\tduek     
\end{array}
\right] 
\left[
\begin{array}{c}
x  \\
y    
\end{array}
\right]= \lambda\pin\left[
\begin{array}{c}
x  \\
y    
\end{array}
\right].
\end{equation}
Starting from  results
in~\cite{bs}, in the next theorem we give
bounds on the eigenvalues which highlight the dependence of the spectrum of $\pin^{-1}\a$
on that of $\sin^{-1}S$.
\vskip5pt
\begin{theorem}\label{eigZ_Cn0} 
Let $\lambda$ and $[x^T, y^T]^T$ be an eigenvalue of problem~(\ref{eig_pb})  
and a corresponding eigenvector. Let
\begin{equation} \label{X}
X = G^{-\frac{1}{2}} (Q+\Theta^{(1)}) G^{-\frac{1}{2}}
\end{equation}
and suppose that $2I_n-X$ is positive definite.
Let 
\begin{eqnarray}
\bar \lambda &=&  \lambda_{\max}(\sin^{-1}S)\,  \max\{  2-\lambda_{\min}(X),\, 1 \} ,
\label{maxl_Z_gen_C} \\
\underline{\lambda} &=&  \lambda_{\min}(\sin^{-1}S)\,  \min\{  2-\lambda_{\max}(X) ,\, 1 \} ,
\label{minl_Z_gen_C}
\end{eqnarray}
if $\tduek \neq 0$, and
\begin{eqnarray}
\bar{\lambda} &=&   \lambda_{\max}(\sin^{-1}S)\, ( 2 -\lambda_{\min}(X) ) \label{maxl_Z_gen},\\
\underline{\lambda} &=&   \lambda_{\min}(\sin^{-1}S)\, (2-\lambda_{\max} (X) )  \label{minl_Z_gen},
\end{eqnarray}
otherwise.

Then, $\pin^{-1}\a$ has at most $2m$ eigenvalues with
nonzero imaginary part, counting conjugates. 
Furthermore, if $\lambda$ has nonzero imaginary part, then 
\begin{equation} \label{boundRe}
\frac{1}{2}(\lambda_{\min}(X)+\underline{\lambda} ) \le \mathcal{R}(\lambda) \le
\frac{1}{2}(\lambda_{\max}(X)+\bar {\lambda}) ,
\end{equation}
otherwise
\begin{eqnarray} 
\min\{\lambda_{\min}(X),\,  \underline{\lambda}\} \le \lambda \le
\max\{\lambda_{\max}(X),\, \bar{ \lambda}\}, \quad \mbox{for } y \neq 0,  \label{bound_yne0} \\
\lambda_{\min}(X)\le \lambda \le \lambda_{\max}(X), \quad \mbox{for } y=0. \qquad\qquad \label{bound_y0}
\end{eqnarray}
Finally, the imaginary part of $\lambda$ satisfies
\begin{equation} \label{boundY}
|\mathcal{I}(\lambda)| \le \sqrt{\lambda_{\max}(S_{inex}^{-1}AG^{-1}A^T)}  \|I_n-X\| .
\end{equation}
\end{theorem}
\vskip -5pt
\begin{proof}
Let 
$$
\sin = RR^T,
$$
be the Cholesky factorization of $\sin$ with  $R$ lower triangular.
With a little algebra,  the eigenvalue problem  (\ref{eig_pb})
becomes
$$   
\left[ 
\begin{array}{cc}
X & Y^T\\
Y & -Z 
\end{array}
\right]
\left[
\begin{array}{c}
 u  \\
 v   
\end{array}
\right]= \lambda
\left[ 
\begin{array}{cc}
I_n &  0 \\
0  & -I_m      
\end{array}
\right]
\left[
\begin{array}{c}
 u  \\
 v    
\end{array}
\right],
$$
where $X$ is the matrix in (\ref{X}) and
\begin{eqnarray}
&& Y = R^{-1}AG^{-\frac{1}{2}}(I_n- X),    \label{Y}\\
&& Z = R^{-1}AG^{-\frac{1}{2}}(2I_n-X)G^{-\frac{1}{2}}A^T   R^{-T} + R^{-1}\tduek R^{-T}, \nonumber \\
&& \left[
\begin{array}{c}
 u  \\
 v    
\end{array}
\right] =
\left[ 
\begin{array}{cc}
G^{\frac{1}{2}} &G^{-1/2}A^T   \\
0  & R^T
\end{array}
\right]
\left[\begin{array}{c}
x \\
y   
\end{array}
\right] \nonumber
\end{eqnarray}
\vspace*{-4mm}
\par \noindent
(see \cite{bs}).

From~\cite[Proposition 2.3]{bs} it follows that  $\pin^{-1}\a$ has at most $2m$ eigenvalues with
nonzero imaginary part, counting conjugates.  Furthermore, \cite[Proposition 2.12]{bs} shows that, when $Y$ is full rank and $Z$ is
positive semidefinite,   if $\lambda$ has nonzero imaginary part, then
\begin{equation} \label{bound_real_complex}
\frac{1}{2}(\lambda_{\min}(X)+\lambda_{\min}(Z))\le \mathcal{R}(\lambda) \le 
\frac{1}{2}(\lambda_{\max}(X)+\lambda_{\max}(Z)) ,
\end{equation}
otherwise 
\begin{eqnarray}
\min\{\lambda_{\min}(X),\,  \lambda_{\min}(Z)\}\le &\lambda& \le 
\max\{\lambda_{\max}(X),\, \lambda_{\max}(Z)\}, \quad \mbox{for } v\neq 0, \label{bound_vne0} \\
\lambda_{\min}(X)\le &\lambda& \le \lambda_{\max}(X) , \quad \mbox{for } v = 0. \label{bound_v0}
\end{eqnarray}
In addition, the imaginary part of $\lambda$ satisfies
\begin{equation} \label{bound_im}
|\mathcal{I}(\lambda)| \le  \|Y\| .
\end{equation}
The same results hold if $Y$ is rank deficient with $Z$ positive definite on ker$(Y^T)$.

By assumptions, $Z$ is positive definite and hence inequalities~(\ref{bound_real_complex})--(\ref{bound_im})
hold. Furthermore, since $v=R^Ty$ and $R$ is nonsingular, (\ref{bound_y0}) is equivalent to~(\ref{bound_v0}).
It remains to prove (\ref{boundRe}), (\ref{bound_yne0}) and (\ref{boundY}); we proceed
by bounding $\lambda_{\min}(Z)$, $\lambda_{\max}(Z)$ and $\|Y\|$.

We consider first  the case $\tduek \ne  0$. In order to provide an upper bound on
$\lambda_{max}(Z)$, we note that
$$
Z=R^{-1}\left[ 
\begin{array}{cc}
AG^{-\frac{1}{2}}&   ({\tduek})^{\frac{1}{2}}
\end{array}
\right]
\left[ 
\begin{array}{cc}
2I_n-X&  0 \\
0  & I_m
\end{array}
\right]
\left[ 
\begin{array}{c}
G^{-\frac{1}{2}}A^T\\
({\tduek})^{\frac{1}{2}} 
\end{array}
\right] R^{-T}.
$$
Let $VU$ be the rank-retaining factorization  of
$\left[ 
\begin{array}{c}
G^{-\frac{1}{2}}A^T\\
({\tduek})^{\frac{1}{2}}
\end{array}
\right] $,
where $V\in \mathbb{R}^{(n+m)\times m}$ has orthogonal columns and $U \in \mathbb{R}^{m\times m}$
is upper triangular and nonsingular. Then, $S=U^TU$ and 
\begin{equation}\label{fattZ}
Z=R^{-1}U^TV^T
\left[ 
\begin{array}{cc}
2I_n-X&  0 \\
0  & I_m
\end{array}
\right] VUR^{-T}.
\end{equation}
Letting $N=R^{-1}U^T$ , we have 
$$
\lambda_{\max}(Z)=\|Z\|\le \|N\|^2\,  \max \{ 2-\lambda_{\min}(X) ,\, 1 \} .
$$
Since $N^TN$ is similar to $\sin^{-1}\, S$,
the upper bounds in (\ref{boundRe}) and (\ref{bound_yne0}),
with $\bar{\lambda}$ given in (\ref{maxl_Z_gen_C}),
follow from the upper bounds in (\ref{bound_real_complex}) and (\ref{bound_vne0}).

In order to provide a lower bound on $\lambda_{\min}(Z)$, we observe that
\begin{equation}\label{lminz}
\frac{1}{\lambda_{\min}(Z)} = \left\|Z^{-1} \right\|
\le \left\|N^{-1} \right\|^2 \left \| \left(V^T
\left[ 
\begin{array}{cc}
2I_n-X&  0 \\
0  & I_m
\end{array}
\right] V \right)^{-1}\right\| .
\end{equation}
By noting that $N^{-T}N^{-1}$ is similar to $S^{-1}\, \sin$
and hence
$$
\left\| N^{-1} \right\|^2  =  \frac{1}{\lambda_{\min}(\sin^{-1} \, S )},
$$
and by applying the Courant-Fischer minimax characterization (see, e.g., \cite[Theorem 8.1.2]{gv})
to the rightmost term in (\ref{lminz}), we have that
the lower bounds in (\ref{boundRe}) and (\ref{bound_yne0}),
with $\underline{\lambda}$ given in (\ref{minl_Z_gen_C}), 
follow from the lower bounds in (\ref{bound_real_complex}) and
(\ref{bound_vne0}).

When $\tduek = 0$, inequalities (\ref{boundRe})--(\ref{bound_yne0})
can be derived by assuming that $VU$ is the rank-retaining factorization of
$G^{-\frac{1}{2}}A^T$,  with $V\in \mathbb{R}^{n\times m}$ having orthogonal columns
and $U \in \mathbb{R}^{m\times m}$ upper triangular and nonsingular.
In this case equality~(\ref{fattZ}) becomes
$$
Z=R^{-1}U^TV^T (2I_n-X) VUR^{-T}
$$
and the thesis follows by reasoning as above.

Finally, from (\ref{Y}) it follows that
\begin{equation} \label{boundY2}
\|Y\|\le \|R^{-1} A G^{-\frac{1}{2}}\| \, \|I_n-X\| = \sqrt{\lambda_{\max}(R^{-1}AG^{-1}A^TR^{-T})} \, \|I_n-X\| .
\end{equation}
Inequality (\ref{boundY}) follows from (\ref{bound_im}) and (\ref{boundY2}) by noting that
$R^{-1}AG^{-1}A^TR^{-T}$ and $\sin^{-1}AG^{-1}A^T$ are similar. 
\end{proof}

\vskip 5pt
\begin{remark}
Note that the assumption on $2I_n-X$ in the theorem can be fulfilled by a proper scaling
of $X$ enforcing its eigenvalues to be smaller than 2.
Furthermore, if $Q$ is diagonal, then $X=I_n$ and
from (\ref{boundY}) it follows that all the eigenvalues of $\pin^{-1}\a$ are real.
In this case $\pin^{-1}\a$
has at least $n$ unit eigenvalues, with $n$ associated independent eigenvectors of the
form $[x^T, 0^T]^T$,
and the remaining eigenvalues lie
in the interval $[\lambda_{\min}(\sin^{-1} S), \, \lambda_{\max} (\sin^{-1} \, S)]$
(see also \cite{dr}).
\end{remark}


\section{Building an inexact constraint preconditioner by updates\label{sec3}}

In this section we design a strategy for updating the CP built 
for some {\em seed} matrix of the KKT sequence. The update is based on 
low-rank corrections and generates  inexact constraint
preconditioners for subsequent systems.

Let us consider a KKT matrix $\as$ generated at some iteration $r$ of 
the IP procedure,
\begin{equation}\label{seed_matrixR}
\as = \left[ 
\begin{array}{cc}
Q+\tunos  & A^T  \\
A   &  -\tdues  
\end{array}
\right],
\end{equation}
where $\tunos\in \mathbb{R}^{n \times n}$ is diagonal positive definite and 
$\tdues\in \mathbb{R}^{m \times m}$ is diagonal positive semidefinite.
The corresponding  {\em seed} CP has the form
\begin{equation}\label{PH}
\ps=\left[ 
\begin{array}{cc}
H  & A^T  \\
A   &  -\tdues      
\end{array}
\right],
\end{equation}
where $H={\rm diag}(Q+\tunos)$. 
Assume that a block factorization of $\ps$ has been obtained by computing
the Cholesky-like factorization of the negative Schur complement, $\ss$, of $H$ in $\ps$,
\begin{equation}\label{SS}
\ss=AH^{-1}A^T+\tdues= LDL^T,
\end{equation}
where $L$ is unit lower triangular and $D$ is diagonal positive definite.

Let $\a$ in (\ref{a}) be a subsequent matrix of the KKT sequence, $\pex$ in (\ref{pexfatt})
the corresponding CP, with $G$ defined as in (\ref{g}), and $S$ the Schur complement (\ref{S}).
We replace $S$ with a suitable update of $\ss$, named  $\sup$,
obtaining the inexact preconditioner  
\begin{eqnarray}\label{pupdated}
\pup&=&
\left[ 
\begin{array}{cc}
I_n  & 0  \\
AG^{-1}   & I_m      
\end{array}
\right]
\left[ 
\begin{array}{cc}
G  &  0  \\
0   &    -\sup  
\end{array}
\right]\left[ 
\begin{array}{cc}
I_n  & G^{-1}A^T  \\
0   & I_m       
\end{array}
\right].
\end{eqnarray}
In the remainder of this section we show that 
specific choices of $\sup$ provide easily computable bounds on the eigenvalues of $\sup^{-1}S$;
we use these bounds, together with the spectral analysis in Section~\ref{sec2}, for constructing
practical inexact preconditioners by updating techniques.
For ease of presentation, we consider  the cases $\tduek=0$ and $\tduek\neq 0$ separately.

\subsection{Updated  preconditioners for $\tduek=0$}

We consider a low-rank update/downdate of $\ss=AH^{-1}A^T=LDL^T$ that generates
a matrix $\sup$ of the form
\begin{equation}\label{sin}
 \sup =AJ^{-1}A^T=L_{upd} D_{upd} L_{upd}^T, 
\end{equation}
where $J$ is a suitably chosen diagonal positive definite matrix,
$L_{upd}$ is unit lower triangular and $D_{upd}$ is diagonal positive definite.

A guideline for choosing $J$ is provided by Theorem~\ref{eigZ_Cn0} and by
a result in~\cite{bsz}, reported next for completeness.
\vskip 5pt 
\begin{lemma}\label{gammabsz}
Let $B\in \mathbb{R}^{m\times n}$ be full rank and let $E, \, F\in \mathbb{R}^{n\times n}$ be symmetric and positive definite.
Then, any eigenvalue $\lambda$ of  $(BEB^T)^{-1} \, BFB^T$ satisfies
\begin{equation}\label{bound_eig_cnon0}
\lambda_{\min}(E^{-1}F)\le \lambda ((BEB^T)^{-1} \, BFB^T ) \le \lambda_{\max}(E^{-1}F) .
\end{equation}
\end{lemma}

For any diagonal matrix $W \in \IR^n$, let $ \gamma(W)=(\gamma_1(W), \ldots,  \gamma_n(W))^T$ 
be the vector with entries given by the diagonal entries of $WG^{-1}$  sorted
in nondecreasing order, i.e.,
\begin{equation}\label{gamma}
\min_{1\le i\le n}  \frac{W_{ii} }{G_{ii}}  = \gamma_1(W) \le \gamma_2(W) \le \cdots\le \gamma_n(W) =  
\max_{1\le i\le  n}  \frac{W_{ii} }{G_{ii}} .
\end{equation}
By Lemma~\ref{gammabsz}, 
the eigenvalues of $\sup^{-1}S$ can be bounded by
using the diagonal entries of $JG^{-1}$, i.e.,
\begin{equation}\label{bound_eig_c0}
\gamma_1(J)\le \lambda (\sup^{-1} S) \le \gamma_n(J);
\end{equation}
then, Theorem~\ref{eigZ_Cn0} implies the following result.

\begin{corollary}\label{corYZtdue0}
Let $\a$,  $\pup$ and $\sup$  be  the matrices in (\ref{a}), (\ref{pupdated}) and (\ref{sin}),
respectively, and $\lambda$ an eigenvalue of $\pup^{-1}\a$.
Let $\gamma_1(J)$ and $\gamma_n(J)$ be defined according to (\ref{gamma})
and let $X$ be the matrix in (\ref{X}). If $2I_n-X$ is positive definite, then
(\ref{boundRe})--(\ref{bound_y0})  hold with 
\begin{eqnarray}
\bar \lambda & \le &  \gamma_n(J) \,   (2-\lambda_{\min}(X) \label{maxl_Z}),\\
\underline{\lambda} & \ge &  \gamma_1(J)  \,  (2-\lambda_{\max}(X)) \label{minl_Z}.
\end{eqnarray}
Furthermore,
\begin{equation}
|\mathcal{I}(\lambda)|\le \sqrt{\gamma_n(J)} \, \|I_n-X\| . \label{boundY_gamma}
\end{equation}
\end{corollary}
%

The key issue is to define $J$ such that a good tradeoff can be achieved between the
effectiveness of the bounds provided by Corollary~\ref{corYZtdue0}
and the cost for building  the factors $L_{upd}$ and $D_{upd}$ in (\ref{sin}).
On the basis of this consideration, following \cite{bsz} we define $\sup$ as a low-rank
correction of $\ss$ of the form
$$
\sup = AH^{-1}A^T+ \bar A \bar K \bar A^T= A \left( H^{-1} + K \right) A^T,
$$
where $K \in  \IR^{n\times n}$ is a diagonal matrix with only $q<n$ nonzero entries
on the diagonal, $\bar K \in \IR^{q\times q}$ is the principal submatrix of $K$ containing
these nonzero entries, $\bar A \in \IR^{m\times q}$ is made of the corresponding
$q$ columns of $A$, and $J^{-1} = H^{-1} + K$ accounts for major changes from $H$ to $G$.
In order to choose $K$, we consider the vector $\gamma(H)$, defined according to (\ref{gamma}),
and the vector of integers $l=(l_1,\ldots,l_n)^T$ such that $l_i$ gives the position of $H_{ii}/G_{ii}$
in $\gamma(H)$, i.e.,
$$\gamma_{l_i}(H)=\frac{H_{ii}}{G_{ii}},$$
and define the set
\begin{equation} \label{Gamma}
\begin{array}{ll}
\displaystyle \Gamma & \displaystyle \!\!\!\! =  \, \{i: n-q_1+1 \le l_i \le n \mbox{ and } \gamma_{l_i}(H) > \mu_\gamma\} \\
                                   & \displaystyle \!\!\!\!  \cup  \;\;  \{i: 1 \le l_i \le q_2 \mbox{ and } \gamma_{l_i}(H) < \nu_\gamma\} ,
\end{array}
\end{equation}
where $\mu_\gamma\ge 1$, $\nu_\gamma \in (0,1]$, and $q_1$ and $q_2$ are nonnegative
integers such that $q_1+q_2=q$. Then we define the diagonal matrix
$K$ by setting
\begin{equation}\label{phi}
K_{ii}=
\left\{
\begin{array}{ll}
  G^{-1}_{ii}-H^{-1}_{ii},   & \mbox{ if } i \in \Gamma ,\\
  0,   & \mbox{ otherwise},
\end{array}\right.
\end{equation}
and choose $\bar K$ as the principal submatrix of $K$ having as diagonal entries the values $K_{ii}$ with index $i \in \Gamma$.
Finally, we define $\bar A$ as the matrix consisting of the columns of $A$ corresponding to the indices in $\Gamma$, ordered as
they are in $A$.


The matrix $\sup$ is positive definite and $\sup-\ss$ is a low-rank matrix
if the cardinality of $\Gamma$ is small, i.e., $q \ll n$.  In this case the Cholesky-like factorization of $\sup$ 
can be conveniently computed by updating or downdating the factorization $LDL^T$ of $\ss$.
Specifically, an update must be performed if $H_{ii} > G_{ii}$, and a downdate if $H_{ii} < G_{ii}$.
This task can be accomplished by either using efficient procedures for updating
and downdating the Cholesky factorization \cite{dh}, 
or by the Shermann-Morrison-Woodbury  formula (see, e.g., \cite[Section~2.1.3]{gv}).
Clearly, once $\sup$ has been factorized, the factorization of $\pup$ is readily available.

The procedure described so far is summarized in Algorithm~\ref{LR} and is called {\em lr\_update}.
It takes in input  the matrices $H$,   $G$ and $A$, the factors $L$ and $D$ of $\ss$,  and the
scalars $q_1$, $q_2$, $\mu_\gamma$, and $\nu_\gamma$, and returns
in output the factors $L_{upd}$ and $D_{upd}$ of $\sup$. 
For simplicity, we assume that the first set on the right-hand side of
(\ref{Gamma})  contains $q_1$ elements and the second set contains $q_2$ elements.
Note that we borrow the Matlab notation.

\begin{algorithm}[t!]
\caption{(Building $S_{upd}$ via low-rank update)\label{LR}}
\begin{minipage}[c]{12.7cm}
\vskip 5pt
\ [$L_{upd}$,$D_{upd}$] =
{\bf {\em lr\_update}} ($H$, $G$, A, $L$, $D$, $q_1$, $q_2$, $\mu_\gamma$, $\nu_\gamma$)
\begin{algorithmic} [1]
\vskip 5pt
\State Build the vector $\gamma(H)$, containing the diagonal entries of $HG^{-1}$ sorted
in nondecreasing order, and the  vector $l$, containing the positions of the diagonal entries
of $HG^{-1}$ in $\gamma(H)$.
\vskip 5pt
\State Build the set $\Gamma$ in (\ref{Gamma}) and the vector $C_\Gamma \in \IR^q$ that contains
the indices $i \in \Gamma$ in increasing order.
\vskip 5pt
\State Build the matrix $\bar K$ by defining its diagonal entries as
$$
\bar K_{ss}=G_{jj}^{-1}-H_{jj}^{-1}, \;\; j=C_\Gamma(s),  \;\; s=1,\ldots,q,
$$
and set $\bar A=A(:,C_\Gamma)$.
\vskip 5pt
\State Compute the factorization $L_{upd} D_{upd} L_{upd}^T$ of
$\sup = \ss + \bar A \bar K \bar A^T$ by updating/downdating
the factors $L$ and $D$. 
\end{algorithmic}
\end{minipage}

\vskip 5pt
\end{algorithm}

The previous choice of $K$ implies that the matrix $J$ has the following diagonal entries:
 $$
J_{ii}=\left\{
\begin{array}{ll}
G_{ii},  \  &\mbox { if } \   i\in \Gamma, \\
H_{ii},  \   &\mbox{ otherwise},
\end{array}\right.
$$
and hence
\begin{eqnarray}
\gamma_1(J)&=&\min\{1, \min_{i\notin \Gamma}\gamma_{l_i}(H) \}= \min\{1, \gamma_{q_2+1}(H)\}, \label{min_g_lr} \\
\gamma_n(J)&=&\max\{1, \max_{i\notin \Gamma}\gamma_{l_i}(H) \}=\max\{1, \gamma_{n-q_1}(H)\}. \label{max_g_lr}
\end{eqnarray}
Then, from Corollary~\ref{corYZtdue0} it follows that, among the possible choices of $J$,
the one considered here is expected to provide good eigenvalue bounds as long as 
$\gamma_{q_2}(H)$ and $\gamma_{n-q_1+1}(H)$ are well separated 
from $\gamma_{q_2+1}(H)$ and $\gamma_{n-q_1}(H)$, respectively.
We observe that by setting $\mu_\gamma=\nu_\gamma=1$ we allow to include in $\Gamma$ indices
corresponding to values $\gamma_{l_i}(H)$ close to 1, which may not change much the bounds
provided by~(\ref{boundY_gamma}) and by (\ref{boundRe})--(\ref{bound_y0}) with (\ref{maxl_Z})--(\ref{minl_Z}). Therefore, the use of larger (smaller) values for
$\mu_\gamma$ ($\nu_\gamma$) may be effective anyway, while saving computational
cost (see Section~\ref{sec:experiments} for further details).
Furthermore, a consequence of this low-rank correction strategy is that $S-\sup$ has $q$ zero eigenvalues;
thus $\pin^{-1}\ak$ has $2q$ unit eigenvalues with geometric multiplicity $q$ \cite[Theorem 3.3]{ss}.

We note that in  the limit case $q=0$ the set $\Gamma$ is empty; 
hence $\sup=\ss$ and 
$$
\gamma_{l_i}(H)  
= \frac{Q_{ii} +(\tunos)_{ii}}{Q_{ii}+\tunok_{ii}}.
$$
The element $\gamma_{l_i}(H)$ is expected to be close to $1$ if $(\tunok)_{ii} - (\tunos)_{ii}$ 
is small, while it may significantly differ from $1$ 
if $\tunok_{ii}$ tends to zero or infinity, as it happens when
the IP iterate approaches an optimal solution where strict complementarity holds.

\begin{figure}[t]
\begin{center}
\hspace*{-5mm}
  \includegraphics[width=.48\textwidth]{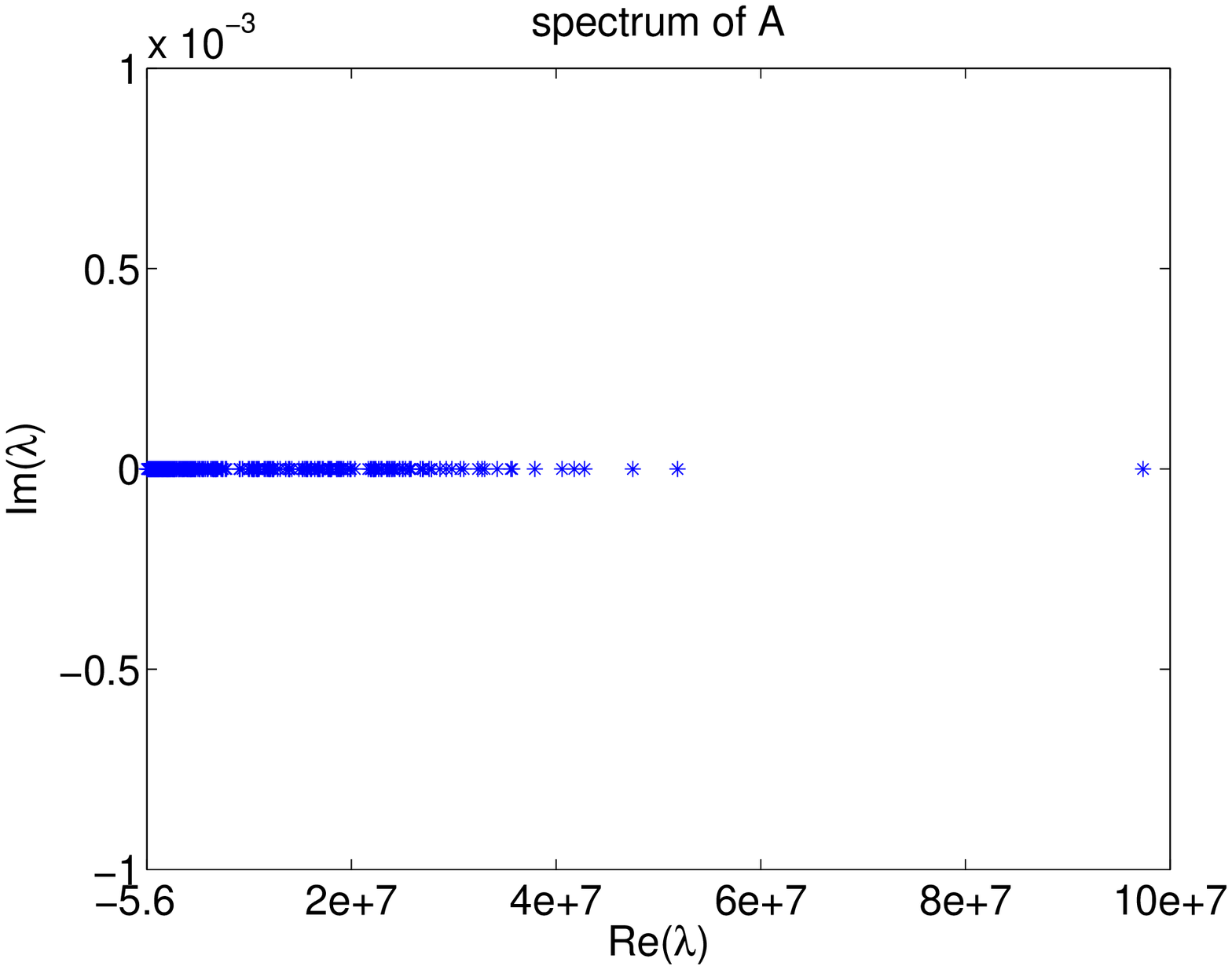}       
  \includegraphics[width=.48\textwidth]{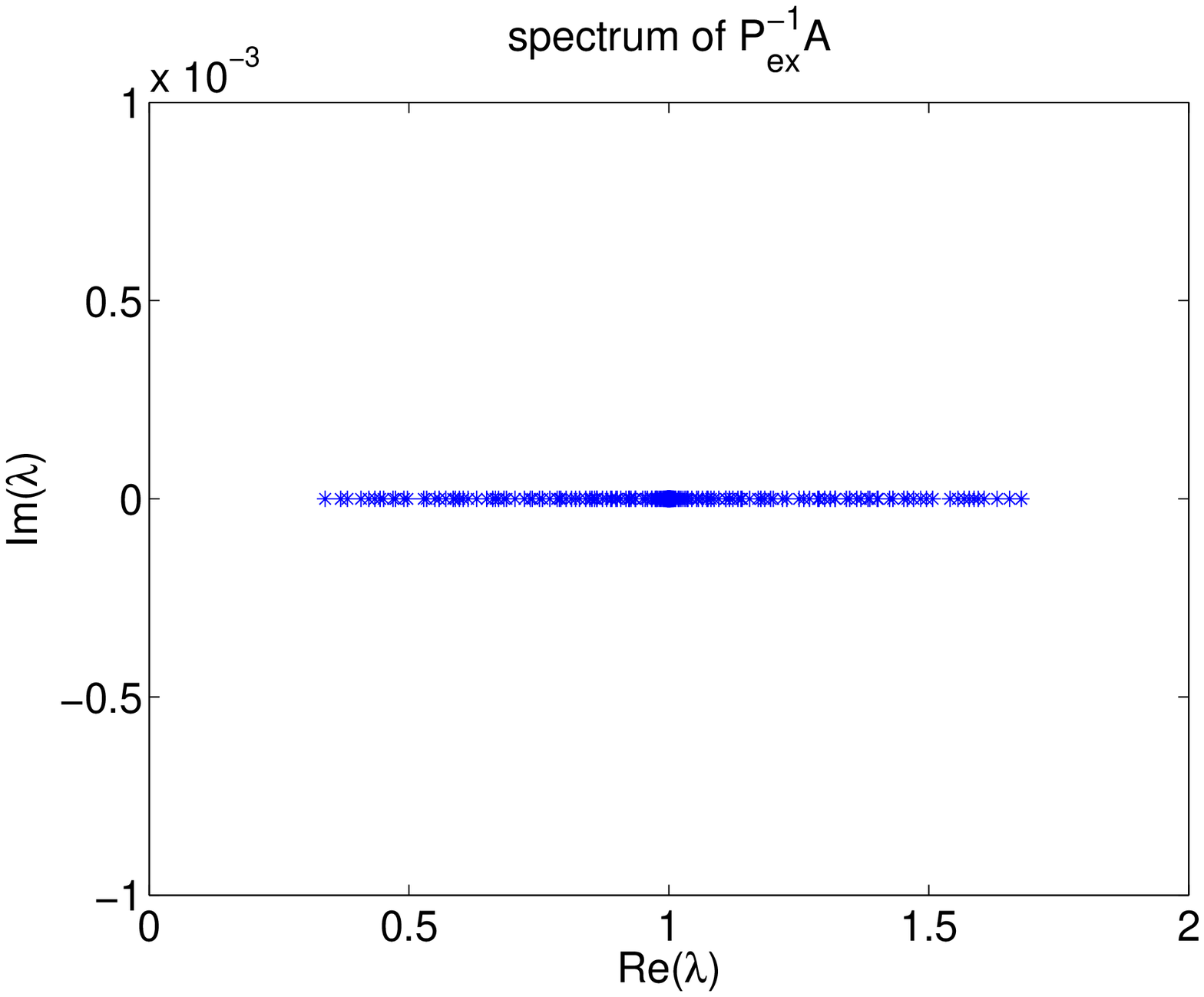} \\[1mm]
\includegraphics[width=.48\textwidth]{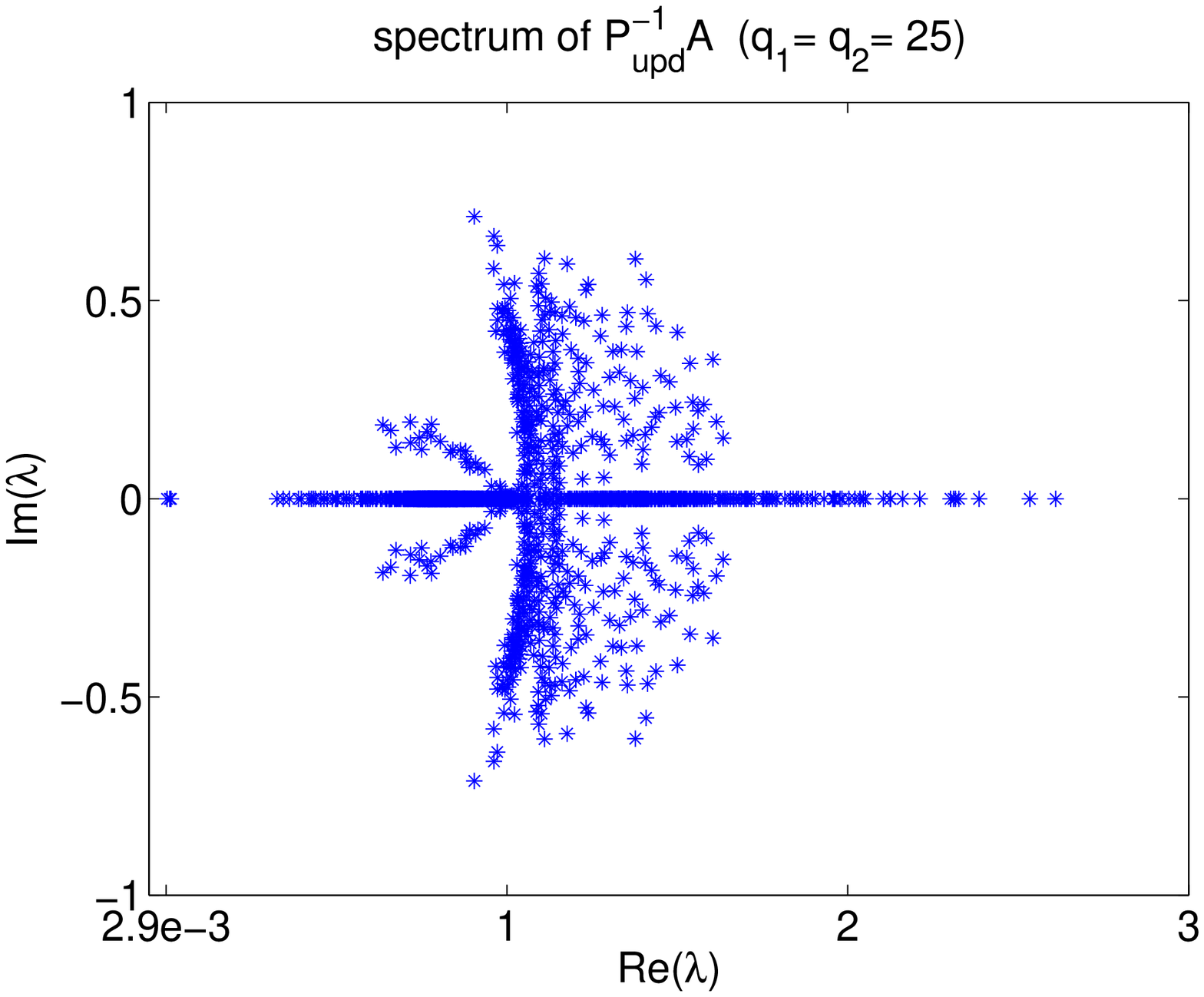}
\end{center}
\caption{Spectra of the matrices $\a$, $\pex^{-1}\a$, and $\pup^{-1}\a$ with $q_1=q_2=25$,
at the 10th iteration of an IP method applied to problem CVXQP1 ($n=1000$, $m=500$). The updated
preconditioner is built from the Schur complement of the exact preconditioner at the 6th IP iteration.
\label{spectra}}
\end{figure}

We conclude this section showing the spectra of the matrices $\a$, $\pex^{-1}\a$ and
$\pup^{-1}\a$, where $\a$ has been obtained by applying an IP solver to problem CVXQP1
from the CUTEst collection \cite{got2}, with dimensions $n=1000$ and $m=500$
(see Figure~\ref{spectra}). The IP solver, described in Section~\ref{sec:experiments},
has been run using $\pex$, and $\a$ is the KKT matrix at the 10th IP iteration.
The updated preconditioner $\pup$ has been built by updating the matrix $\as$
obtained at the 6th IP iteration, using $q_1=q_2=25$. In this case
$\gamma_{q_2+1} = 1.85$e-3 and $\gamma_{n-q_1} = 2.73$e+0.
We see that, unlike $\pex$, $\pup$ moves some eigenvalues from the real to
the complex field. Nevertheless, $\pup$ tends to cluster the eigenvalues of $\a$ around 1, and
$\gamma_{q_2+1}$ and $\gamma_{n-q_1}$ provide approximate bounds on the real and imaginary
parts of the eigenvalues of the preconditioned matrix, according to
(\ref{boundY_gamma})--(\ref{minl_Z}) and (\ref{min_g_lr})--(\ref{max_g_lr}).
Of course, this clustering is less effective than the one performed by $\pex$, but it is useful
in several cases, as shown in Section~\ref{sec:experiments}.

\subsection{Updated  preconditioners for $\tduek\neq0$}

The updating strategy described in the previous section can be generalized to the case $\tduek \ne 0$.
To this end, we note that the sparsity pattern of  $\Theta^{(2)}$ does not change   
throughout the IP iterations and the set   ${\mathcal L}=\{i \; : \; \Theta^{(2)}_{ii}\ne 0\}$
has cardinality equal to the number $m_1$ of linear inequality constraints in the QP problem. Let 
$\tilde \Theta_{seed}^{(2)}$ and $\tilde \Theta^{(2)}$  be the
$m_1 \times m_1$ diagonal submatrices containing the nonzero diagonal entries of
$\Theta_{seed}^{(2)}$ and $\Theta^{(2)}$, respectively, and let $\tilde I_m$ be the rectangular matrix
consisting of the columns of $I_m$  with indices in $ {\mathcal L}$.  
Then, we have
\begin{eqnarray*}
\ss &=  AH^{-1}A^T +\Theta_{seed}^{(2)} &= \tilde A \tilde H^{-1} \tilde A^T,\\
S &=  AG^{-1}A^T +\Theta^{(2)} &=  \tilde A \tilde G^{-1} \tilde A^T,
\end{eqnarray*}
where 
$$
\tilde A   =   \left[\begin{array}{ll}
  A & \tilde I_m
\end{array} \right] , \quad \tilde H^{-1} =   \left[\begin{array}{ll}
  H^{-1}  & 0\\
  0 & \tilde \Theta_{seed}^{(2)}
\end{array} \right], \quad 
\tilde G^{-1} = 
\left[\begin{array}{ll}
  G^{-1}  & 0\\
  0 & \tilde \Theta^{(2)}
\end{array} \right].
$$
%
%
Analogously, letting
\begin{equation}\label{sin2}
S_{upd}=AJ^{-1}A^T +\Theta_{upd}^{(2)} , 
\end{equation}
we have 
$$
\sup= \tilde A \tilde J^{-1}\tilde A^T, \quad \mbox{ where } \quad 
\tilde J^{-1}= \left[\begin{array}{ll}
  J^{-1}  & 0\\
  0 & \tilde \Theta^{(2)}_{upd}
\end{array} \right],
$$
and $\tilde \Theta^{(2)}_{upd}$ is the $m_1 \times m_1$ diagonal submatrix of
$\Theta_{upd}^{(2)}$ containing its nonzero diagonal entries. 
Thus,  we can choose $\tilde J$ using the same arguments as in the previous section.
With a little abuse of notation,
let $\gamma(\tilde J)=(\gamma_1(\tilde J), \ldots, \gamma_{n+m_1}(\tilde J))$ 
be the vector with elements equal to the diagonal entries
of $\tilde J \tilde G^{-1}$ sorted in nondecreasing order.
Then, by Lemma~\ref{gammabsz}, the eigenvalues of  $S_{upd}^{-1} S $ satisfy
\begin{equation}\label{bound_gamma_tilde}
\gamma_1(\tilde J) \le \lambda (S_{upd}^{-1} S) \le \gamma_{n+m_1}(\tilde J),
\end{equation}
and the following result holds.

\begin{corollary}\label{corYZtduenonzero}
Let $\a$,  $\pup$ and $\sup$ be  the matrices in (\ref{a}), (\ref{pupdated}) and (\ref{sin2}),
respectively, and $\lambda$ an eigenvalue of $\pup^{-1}\a$.
Let $\gamma_1(\tilde J)$ and $\gamma_{n+m_1}(\tilde J)$ be
the smallest and the largest element of $\gamma(\tilde J)$ and let $X$ be the matrix in (\ref{X}).
If $2I_n-X$ is positive definite, then  (\ref{boundRe})--(\ref{bound_y0}) hold with
\begin{eqnarray}
\bar{\lambda} & \le & \gamma_{n+m_1}(\tilde J)\, \max\{  2-\lambda_{\min} (X) ,\, 1 \} \label{maxl_Z_C},\\
\underline{\lambda} & \ge & \gamma_1(\tilde J) \,  \min\{  2-\lambda_{\max} (X) ,\, 1 \}\label{minl_Z_C}.
\end{eqnarray}
Furthermore,
\begin{equation}\label{boundY_gamma_2}
|\mathcal{I}(\lambda)|\le \sqrt{\gamma_{n+m_1}(\tilde J)} \, \|I_n-X\| .
\end{equation}
\end{corollary}
\begin{proof}
Inequalities (\ref{maxl_Z_C}) and (\ref{minl_Z_C}) follow directly from Theorem \ref{eigZ_Cn0} and
(\ref{bound_gamma_tilde}).
Since $\tduek$ is  positive semidefinite, for any vector $w\in \mathbb{R}^n$ we have
$$
w^T \sup^{-\frac{1}{2}} AG^{-1}A^T \sup^{-\frac{1}{2}} w 
\le  w^T \sup^{-\frac{1}{2}} (AG^{-1}A^T +\tduek) \sup w.
$$
Then, by matrix similarity,
\begin{equation}\label{lmaxY}
\lambda_{\max}(\sup^{-1}AG^{-1}A^T)  \le \lambda_{\max}(\sup^{-1} S)
\end{equation}
and (\ref{boundY_gamma_2}) follows by  using Theorem  \ref{eigZ_Cn0} and (\ref{bound_gamma_tilde}).
\end{proof}
\vskip 5pt

On the basis of the previous results, the generalization of the updating procedure to the case
$\tduek\neq 0$ is straightforward. If $S_{seed}=LDL^T$, the factorization of
$\pup$ can be computed by  invoking procedure {\em lr\_update} with input data $\tilde H$,
$\tilde G$ and $\tilde A$ in place of $H$,  $G$ and $A$.

\section{Numerical results\label{sec:experiments}}

We tested the effectiveness of our updating procedure by solving sequences of KKT systems
arising in the solution of convex QP problems where $m \le n$ and $A$ is full rank.

To this end, we implemented {\em lr\_update} within PRQP, a Fortran~90 solver
for convex QP problems based on a primal-dual inexact Potential Reduction IP
method \cite{cddd_coap1, cdddt, ddd}.
For comparison purpose, we implemented the preconditioner $\pex$ too,
in the form specified in (\ref{pexfatt}). We used the CHOLMOD
library \cite{dh} to compute the sparse $LDL^T$ factorization of $\ss$
and $S$, and to perform the low-rank updates and downdates required by $\sup$.
For the solution of the KKT systems, we developed an implementation of
the left-preconditioned SQMR method without look-ahead \cite{fn},
taking into account the block structure of the system matrices and of the exact and
updated preconditioners. All the new code was written in Fortran 90, with interfaces to
the functions of CHOLMOD, written in C.
We note that only one matrix-vector product per iteration is
performed in our SQMR implementation, except in the last few iterations, where an additional
matrix-vector product per iteration is computed to use the residual instead of the preconditioned
BCG-residual in the stopping criterion, as in the QMRPACK code \cite{fn_qmrpack}.
This keeps the computational cost per iteration comparable with that of the
Conjugate Gradient method. We also observe that, although no theoretical convergence
estimates are available for SQMR, this method has shown good performance in all our
experiments. 

PRQP was run on several test problems, either taken from the CUTEst collection~\cite{got2}
or obtained by modifying CUTEst problems, as explained later in this section.
The starting point was chosen as explained in \cite{ddd2} and the IP iterations
were stopped when the relative duality gap and suitable measures of the
primal and dual infeasibilities became lower than $10^{-7}$ and $10^{-8}$,
respectively (see \cite{cddd_coap2}
for the details). The zero vector was used as starting guess in SQMR.
An adaptive criterion was applied to stop the iterations~\cite{cddd_coap2},
which relates the accuracy in the solution of the KKT system to the quality of the current
IP iterate, in the spirit of inexact IP methods~\cite{b_jota}.
A maximum number of 1000 SQMR iterations was considered too, but it was never
reached in our runs.

Concerning the choice of $q$, some comments are in order. 
As the value of $q$ increases, the updated preconditioner is expected to improve
its effectiveness, reducing the number of linear iterations. On the other hand, its
computational cost is also expected to increase, because of the growing
cost of the low-rank modification. In order to reduce the time for the solution
of the overall KKT sequence, the updating strategy must realize a tradeoff between effectiveness
and cost, and our experience has shown that $q$ must be much smaller than the dimension
of the Schur complement. On the basis of these considerations, we also think that
heuristic rules for choosing $q$ dynamically, such as the one
proposed in \cite{wol}, may have limited impact of our procedure because only a 
small range of values of $q$ is affordable. However, we postpone a systematic study
of this issue to future work. 
In the experiments we set $q=50,100$ and $q_1=q_2=q/2$, which is much smaller
than the dimension of the Schur complement in our test problems.
We also considered ``the limit case'' $q=0$, corresponding to $\sup = \ss$
in the updated preconditioner $\pup$. Preliminary experiments with larger values
of $q$, i.e., $q=150,200$, did not lead to any performance improvement and
therefore we decided to discard these values. Furthermore, since we had not
obtained practical benefits by including in $\Gamma$ indices corresponding
to values of $\gamma_{l_i}(H)$ close to~1, we set $\mu_\gamma = 10$ and
$\nu_\gamma=0.1$. When the number of elements $\gamma_{l_i}(H) > \mu_\gamma$
or the number of elements $\gamma_{l_i}(H) < \nu_\gamma$ was
less than $q/2$, we chose $q_1$ and $q_2$ to get the largest possible value
of $q_1+q_2$. The same comments hold for $\gamma_{l_i}(\tilde H)$.
For simplicity, in the following the notation $\gamma_{l_i}(H)$ is used also
to indicate $\gamma_{l_i}(\tilde H)$, as it will be clear from the context.

On the basis of numerical experiments, we decided to refresh the preconditioner,
i.e., to build $\pex$ instead of $\pup$, when the time for computing $\pup$
and solving the linear system exceeded 90\% of the time for building the last exact
preconditioner and solving the corresponding system.
When for a specific system of the sequence this situation occured, the next system
of the sequence was solved using the preconditioner $\pex$. We also set a maximum
number, $k_{max}$, of consecutive preconditioner updates, after which the refresh
was performed anyway. This strategy aims at avoiding possible situations in which
the time saved by updating the Schur complement, instead of re-factorizing it,
is offset by an excessive increase in the number of SQMR iterations, due to
deterioration of the quality of the preconditioner.
In the experiments discussed here $k_{max}=5$ was used for all the test problems.

We observe that we did not apply any scaling to the matrix 
$X$ in (\ref{X}), although our
supporting theory lies on the assumption that the eigenvalues of $X$ are smaller
than 2. Nevertheless, the results generally appear to be
in agreement with the theory.
It is also worth noting that since each KKT system is built from the approximation of
the optimal solution computed at the previous IP iteration, different preconditioners
produce changes in the sequence of KKT systems. On the other hand, testing our
updating technique inside an IP solver allows to better evaluate its impact
on the performance of the overall optimization method, providing a more
general evaluation of our approach. In our experiments
we did not observe strong differences in the behaviour of the IP method
by using $\pex$ or $\pup$ with different values of $q$, and the number
of IP iterations was generally unaffected by the choice of the preconditioner
(a small variation of the IP iterations was observed in few cases,
as shown in the tables reported in the next pages). 

We performed the numerical experiments on an Intel Core 2 Duo E7300 processor
with clock frequency of 2.66 GHz, 4 GB of RAM and 3 MB of cache memory,
running Debian GNU/Linux 6.0.7 (kernel version 2.6.32-5-amd64).
All the software was compiled using the the GNU C and Fortran compilers (version 4.4.5).

Since the performance of the updating strategy was expected to depend
on the cost of the factorization of the Schur complement, we selected
test problems requiring different factorization costs.
As a first test set, we considered some problems
taken from the CUTEst collection \cite{got2} or obtained
by modifying problems available in this collection, as explained next.
Most of the large CUTEst convex QP problems with inequality constraints
(corresponding to $\tduek \ne 0$) were not useful for our experiments,
because of the extremely low cost of the factorization of their Schur complements.
Therefore, we also modified two CUTEst QP problems with linear constraints $Ax=b$
and non-negligible factorization costs, by changing $Ax=b$ into $Ax \ge b$. 
The problems of this test set are listed in Table~\ref{table1},
along with their dimensions and the number of nonzero entries of their Schur complements.
The modified problems are identified by appending ``-M'' to their original names.
In the first three problems $\tduek = 0$, while in the remaining ones $\tduek \ne 0$. 
For all the problems the Schur complements are very sparse; furthermore, they are 
banded for STCQP2, MOSARQP1 and QPBAND (diagonal for the latter problem).

\setlength{\tabcolsep}{5pt}

\begin{sidewaystable}[p!]
\begin{center}
{\small
\begin{tabular}{|c|c|ccc|ccc|ccc|ccc|}
\hline
& &
\multicolumn{3}{|c|}{$\pex$} & 
\multicolumn{3}{|c|}{$\pup$ ($q=0$)}  &
\multicolumn{3}{|c|}{$\pup$ ($q$=50)} &
\multicolumn{3}{|c|}{ \ $\pup$ ($q$=100) \ }\\
\cline{3-14}
{\em Problem} & $\begin{array}{c}  n \\ m \\ nnz(S) \\[1mm] \end{array}$
                       & $I\!Pits$ & $its$ &  $time$ & $I\!Pits$ & $its$ & $time$ & $I\!Pits$ & $its$ & $time$ & $I\!Pits$ & $its$ & $time$ \\
\hline
CVXQP1 & $\begin{array}{cc} 20000 \\ 10000 \\67976\end{array}$ 
& 16  & 209 & 2.07e+0 & 16 &  298 & 2.35e+0 & 16 &   335 &   2.55e+0   & 16 & 323 & 2.49e+0 \\ \hline
CVXQP3 &  $\begin{array}{cc} 20000 \\ 15000 \\ 155942\end{array}$ 
             & 35 & 523 & 8.04e+0 &  35 & 800 & 9.32e+0 & 35  & 755 &   8.86e+0   &  35 & 757  & 8.90e+0 \\ \hline
STCQP2 & $\begin{array}{cc} 16385 \\ 8190 \\ 114660 \end{array}$
             &12  &226  & 1.46e+0  &  12 &  235  & 1.41e+0  &  12 &  235 & 1.43e+0  &   12 & 235 & 1.43e+0 \\ \hline
CVXQP1-M & $\begin{array}{cc} 20000 \\ 10000 \\ 67976\end{array}$ 
            &26   & 1015 &  7.65e+0 & 29 & 1562 & 1.07e+1 & 26 &  1812 &    1.21e+1    & 26 &  1845 & 1.22e+1 \\ \hline
CVXQP3-M &  $\begin{array}{cc} 15000 \\ 11250 \\ 155942\end{array}$ 
            & 30 & 1261 & 1.47e+1 & 30 & 1654 & 1.71e+1 & 30 & 2073 &   2.11e+1   & 30 &  2135 & 2.18e+1 \\ \hline
MOSARQP1 & $\begin{array}{cc} 22500 \\ 20000  \\   257166\end{array}$
             &  16 & 66  &    4.65e+0  & 16 &  193  &   4.53e+0   & 16 &  189 & 4.83e+0  & 16 &   215 & 5.21e+0 \\ \hline			
QPBAND & $\begin{array}{cc}50000  \\  25000 \\  25000 \end{array}$ 
            & 12 & 757	& 7.13e+0	& 12 & 1596	&	1.37e+1	&  12 & 1600	&	1.43e+1 &	12 &  1599	&1.37e+1\\ \hline
\end{tabular}
}
\medskip
\caption{Comparison between $\pex$ and $\pup$ on the first set of test problems.\label{table1}}
\end{center}

\end{sidewaystable}

A comparison among the exact and updated preconditioners on this set of problems
is presented in Table~\ref{table1}. For each preconditioner we report the total number of
PRQP iterations ($I\!Pits$), the total number of SQMR iterations ($its$) and the
overall computation time, in seconds, needed to solve the KKT sequence. 
The results shows that using the updating strategy is not beneficial on these problems.
Since the exact factorization of the Schur complement is not significantly more expensive
than SQMR, the time saved by applying the updating strategy is not enough to offset
the time required by the larger number of SQMR iterations
resulting from the use of an approximate CP. Nevertheless, for STCQP2 and MOSARQP1,
the updating strategy shows the same performance as the exact preconditioner.
We also see that in many cases
the number of SQMR iterations increases with $q$, which seems contrary to
our expectations. This behaviour has been explained with a deeper analysis of the
execution of PRQP on the selected test problems. 
The convergence histories of the IP method enlighten that the IP iterations
at which the refresh takes place vary with $q$; this does not allow a fair comparison
among the updating rules using different values of the parameter $q$. This behaviour
is ascribed to the fact that the computation of the exact preconditioner is not expensive
and therefore the refresh strategy is very sensitive to the choice of $q$.
In particular, in our experiments the choice $q=0$ often leads to recomputing
the exact preconditioners before the number of SQMR iterations increases too much,
thus reducing the iteration count with respect to larger values of $q$.
We also note that in some cases the number of
SQMR iterations is practically constant as $q$ varies, because either the number of
elements $\gamma_{l_i}(H) \not \in  [\nu_\gamma, \mu_\gamma]$
is much smaller than $q$, or the values $\gamma_{l_i}(H)$ excluded by the updating
strategy are not well separated from $\gamma_{q_2}(H)$ and $\gamma_{n-q_1+1}(H)$.

Despite these first unfavourable results, since the updating strategy does not
excessively increase the number of SQMR iterations,
we can still expect a significant time reduction
on problems with Schur complements requiring large factorization times.
In order to investigate this issue, we built a second set of test problems with less sparse
Schur complements, by modifying the problems in Table~\ref{table1} as follows.
In problems CVXQP1, CVXQP1-M, CVXQP3 and CVXQP3-M, we added four nonzero
entries per row in the matrix $A$, while in problem QPBAND we added two nonzero
entries per row.
In problem  MOSARQP1, we introduced nonzeros in the positions $(i,n)$ of the constraint matrix $A$,
where $i$ is such that $mod(i,10)=1$. These new problems are identified by appending
``-D'' to the names of the problems they come from, as listed in Table~\ref{table2}
(``D'' stands for ``denser'').
Finally, starting from problems CVXQP3 and CVXQP3-M we generated two further problems,
named CVXQP3-D2 and CVXQP3-M-D2, respectively. They were obtained by adding only
one nonzero entry per row in the matrix $A$. The densities of the resulting Schur complements are between
the Schur complements densities of the corresponding original and -D versions. 

\begin{sidewaystable}[p!]
\begin{center}
{\small
\begin{tabular}{|c|c|ccc|ccc|ccc|ccc|}
\hline
 & & 
\multicolumn{3}{|c|}{$\pex$} & 
\multicolumn{3}{|c|}{$\pup$ ($q=0$)}  &
\multicolumn{3}{|c|}{$\pup$ ($q$=50)} &
\multicolumn{3}{|c|}{ \ $\pup$ ($q$=100) \ }\\
\cline{3-14}
{\em Problem} & $\begin{array}{c}  n \\ m \\ nnz(S) \\[1mm] \end{array} $ 
             & $I\!Pits$ & $its$ &  $time$ & $I\!Pits$ & $its$ & $time$ & $I\!Pits$  & $its$ & $time$ &$I\!Pits$  & $its$ & $time$ \\
\hline
CVXQP1-D & $\begin{array}{cc} 20000 \\ 10000 \\ 240494 \end{array} $
 & 15 &  239 & 2.95e+2 & 15 & 759 & 9.83e+1 & 15 & 616 &   9.91e+1   & 15 &  602 & 1.03e+2 \\ \hline
CVXQP3-D &  $\begin{array}{cc} 20000 \\ 15000 \\ 542296 \end{array}$ 
            &15   &  192 & 1.03e+3 &  15 & 778 & 3.30e+2 & 15 &  526 &   4.40e+2   &  15 & 481 & 4.55e+2 \\ \hline
CVXQP3-D2 &  $\begin{array}{cc} 20000 \\ 15000\\ 224396 \end{array}$ &
            15   & 288 & 9.95e+1 &  18 & 1009 & 6.02e+1 & 17 &  819 &   5.26e+1   &  17 & 802 & 5.46e+1 \\ \hline
STCQP2-D & $\begin{array}{cc} 16385 \\ 8190 \\ 5003908 \end{array}$   
          & 12   &238  & {  6.08e+2}  & 12 & 262  & 1.22e+2  & 12 &   262 & 1.22+2  & 12 &  262 & 1.22e+2 \\ \hline
CVXQP1-M-D & $\begin{array}{cc} 20000 \\ 10000 \\ 240494 \end{array}$ 
         &28   & 1090& 5.85e+2 & 28& 4704 & 3.63e+2 & 28 & 3665 &    3.23e+2    & 27 & 3514 & 3.24e+2 \\ \hline
CVXQP3-M-D &  $\begin{array}{cc} 20000 \\ 15000 \\ 542296\end{array}$ 
      &25        & 910 & 1.93e+3 & 25 & 3605 & 9.08e+2 & 25 & 3416 &   8.89e+2   & 25 &  3317 & 9.07e+2 \\ \hline
CVXQP3-M-D2&  $\begin{array}{cc} 20000 \\ 15000 \\224396\end{array}$ 
      &25         & 822 & 1.66e+2 & 25 & 2782 & 1.32e+2 & 25 & 2645 &   1.33e+2   & 25 &  2148& 1.25e+2 \\ \hline
MOSARQP1-D & $\begin{array}{cc} 22500 \\ 20000 \\ 573216 \end{array}$ 
   &24          & 93  &    4.94e+1  & 23 &  881  &   3.47e+1   & 22&  599 & 3.00e+1  & 22&  440 & 2.78e+1 \\ \hline
QPBAND-D & $\begin{array}{cc}50000  \\  25000 \\ 149988\end{array}$ 
     &11    & 717	& 1.06e+3	&11 &2614	&	4.26e+2	& 11 & 2619	&	4.36e+2 &	11&  2612  &4.51e+2\\ \hline			
\end{tabular}
}
\medskip
\caption{Comparison between $\pex$ and $\pup$ on the second set of test problems.\label{table2}}
\end{center}

\end{sidewaystable}

The results in Table~\ref{table2} show that when the Schur complement is denser, the updating
procedure provides a significant reduction in the overall computation time, because the
increase in the SQMR iterations is largely offset by the time saving obtained by updating
the factors of the Schur complement instead of recomputing them.
For the problems under consideration the reduction ranges from 21\%, for CVXQP3-M-D2 with
$q=0$, to 80\%, for STCQP2-D with all the three values of $q$. Comparing the behaviour
of the updating strategy on CVXQP3-M-D and CVXQP3-M-D2, we see that the percentage
of time saved with the updating strategy drops from 53-54\% to 21-24\% when going from
CVXQP3-M-D to CVXQP3-M-D2 (the latter has a sparser Schur complement).
A similar behaviour can be observed by comparing the results obtained on
CVXQP3-D and CVXQP3-D2. In this case the best time reduction for CVXQP3-D
amounts to 68\% ($q=0$), while the best one for CVXQP3-D is 47\% ($q=50$).
Furthermore, the time reduction also holds when the number of IP iterations corresponding
to the updating strategy is greater than the number of IP iterations obtained with the
exact preconditioner (see CVXQP3-D2). Surprisingly, also the reverse may happen,
i.e., the updating procedure may slightly reduce the number of IP iterations
(see CVXQP1-M-D and MOSARQP1-D).

We further note that the number of iterations obtained with $\pup$ generally decreases
as $q$ increases; thus, for the second set of problems, updating the Schur complement
by low-rank information appears to be beneficial in terms of iterations.
There are also some cases where the number of SQMR iterations
is practically constant as $q$ varies (see STCQP2-D and QPBAND-D).
In these cases, as for the corresponding problems in the first test set, we verified that
either the number of elements $\gamma_{l_i}(H)$ with indices in $\Gamma$
is very small or even zero, or those values of $\gamma_{l_i}(H)$ are not
well separated from the remaing ones, thus making the low-rank modification ineffective.
For similar reasons the reduction of the number of iterations
from $q=50$ to $q=100$ is generally less significant than from $q=0$ to $q=50$.
Finally, the best results in terms of execution time are mostly obtained with $q>0$.

To provide more insight into the behaviour of the updated preconditioners,
in Tables~\ref{tCVXQP3}-\ref{tMOSARQP1-D} we show some details concerning
the solution of the sequences of KKT systems arising from four problems, i.e.,
CVXQP3, CVXQP3-D, MOSARQP1 and MOSARQP1-D.
For each IP iteration we report the number, $its$, of SQMR iterations,
as well as the time, $T_{prec}$, for building the preconditioner, the time, $T_{solve}$,
for solving the linear system, and their sum, $T_{sum}$. The last row contains
the total number of SQMR iterations and the total times, over all IP iterations, while
the rows in bold correspond to the IP iterations at which the preconditioner
is refreshed. These tables clearly support the previous observation that the updating strategy is
efficient when the computation of $\pex$ is expensive, as it is for CVXQP3-D and  MOSARQP1-D.
Conversely, when the time for building $\pex$ is modest, recomputing $\pex$ is a natural choice.
It also appears that the refresh strategy plays a significant role in achieving
efficiency, since it prevents the preconditioner from excessive deterioration.
Finally, when the time for computing $\pex$ is not dominant, the refresh tends to occur
more frequently, since a small increase in the number of iterations obtained with $\pup$
may easily raise the execution time over 90\% of the time corresponding to the last
application of the exact preconditioner.

\begin{table}
{\small
\begin{center}
\begin{tabular}{|c|rccc|rrrr|}
\hline
 &
\multicolumn{4}{|c|}{$\pex$} & 
\multicolumn{4}{|c|}{$\pup$  ($q$=50)}\\
\cline{2-9}
$I\!P \; it$ & $its$ & $T_{prec}$ & $T_{solve}$ & $T_{sum}$  &  $its$ & $T_{prec}$ & $T_{solve}$ & $T_{sum}$ \\
\hline
 {1} & 23& 8.82e-2& 2.50e-1& 3.38e-1   &   {\bf 23}  & {\bf 8.79e-2}& {\bf 2.49e-1} & {\bf 3.37e-1}\\
 {2} &     8& 7.86e-2& 9.17e-2& 1.70e-1         &           15&    2.21e-2& 1.70e-1& 1.92e-1\\
  {3} &    6& 7.06e-2& 7.35e-2& 1.44e-1            &        20&  2.17e-2& 2.17e-1& 2.39e-1\\
   {4} &   5& 7.86e-2& 6.16e-2& 1.40e-1           &         28&     1.92e-2& 2.93e-1& 3.12e-1\\
  {5} &    5& 7.86e-2& 6.16e-2& 1.40e-1            &        {\bf 5}&    {\bf  7.45e-2} &{\bf  5.85e-2} & {\bf 1.33e-1}\\
  {6} &    5& 7.46e-2& 5.90e-2& 1.34e-1             &       9&     5.54e-3& 9.98e-2& 1.05e-1\\
   {7} & 7& 7.86e-2& 8.39e-2& 1.62e-1             &       17&    1.37e-2& 1.84e-1& 1.98e-1\\
   {8} &   7& 7.46e-2& 8.00e-2& 1.55e-1             &       {\bf 7}&     {\bf 7.46e-2}& {\bf 7.96e-2}& {\bf 1.54e-1}\\
  9 &   9& 7.46e-2& 1.01e-1& 1.75e-1               &     11&     1.05e-2& 1.20e-1& 1.31e-1\\
   {10} &  9& 7.86e-2& 9.81e-2& 1.77e-1         &            20&     1.48e-2& 2.01e-1& 2.16e-1\\
 {11} &   11& 7.86e-2& 1.19e-1& 1.98e-1          &       {\bf  11}&   {\bf  7.46e-2} & {\bf 1.19e-1} & {\bf 1.93e-1}\\
 {12} &   12& 7.46e-2& 1.31e-1& 2.05e-1          &           14&    1.05e-2&  1.42e-1& 1.52e-1\\
 {13} &   12& 7.86e-2& 1.26e-1& 2.04e-1           &          34&   1.88e-2& 3.34e-1& 3.53e-1\\
 {14} &   12& 7.86e-2& 1.25e-1& 2.04e-1               &      {\bf 12}&     {\bf 7.45e-2}& {\bf 1.24e-1}& {\bf 1.99e-1}\\
 { 15} &   12& 7.06e-2& 1.25e-1& 1.96e-1      &               16&     1.06e-2& 1.59e-1& 1.70e-1\\
 { 16} &   12& 7.86e-2& 1.26e-1& 2.04e-1             &        28&     1.17e-2& 2.72e-1& 2.84e-1\\
 {17} &   14& 7.86e-2& 1.47e-1& 2.25e-1            &         {\bf 14}& {\bf  7.47e-2} & {\bf 1.45e-1}& {\bf 2.20e-1}\\
 {18} &   14& 7.46e-2& 1.44e-1& 2.19e-1          &           20&     1.04e-2& 1.99e-1& 2.10e-1\\
   {19} & 14& 7.86e-2& 1.44e-1& 2.23e-1            &      {\bf  14}& {\bf  7.06e-2}& {\bf 1.43e-1} & {\bf 2.13e-1}\\
 { 20} &   14& 7.46e-2& 1.41e-1& 2.16e-1      &               21&     1.06e-2& 2.01e-1& 2.12e-1\\
 {21} &   16& 7.86e-2& 1.65e-1& 2.43e-1       &         {\bf 16}& {\bf 7.46e-2}& {\bf 1.63e-1} & {\bf 2.37e-1}\\
 {22} &   14& 7.86e-2& 1.43e-1& 2.22e-1           &          23&     1.02e-2& 2.31e-1& 2.41e-1\\
\vdots & \vdots & \vdots & \vdots & \vdots & \vdots & \vdots & \vdots & \vdots  \\
  {34} &  28& 7.46e-2& 2.68e-1& 3.43e-1         &          {\bf  28}&  {\bf  7.46e-2} & {\bf 2.65e-1} & {\bf 3.40e-1}\\
  {35} &   28& 7.46e-2& 2.71e-1& 3.45e-1          &        42&     1.14e-2& 3.97e-1& 4.08e-1\\
\hline
&   523& 2.69e+0 & 5.36e+0& 8.04e+0    &                      755&      1.33e+0& 7.53e+0& 8.86e+0 \\
         \hline
\end{tabular}
\end{center}
}
\medskip
\caption{CVXQP3: details for $\pex$ and $\pup$ with $q=50$.\label{tCVXQP3}}
\end{table}

\begin{table}
{\small
\begin{center}
\begin{tabular}{|c|rccc|rrrr|}
\hline
 &
\multicolumn{4}{|c|}{$\pex$} & 
\multicolumn{4}{|c|}{$\pup$  ($q$=50)}\\
\cline{2-9}
$I\!P \; it$ & $its$ & $T_{fact}$ & $T_{solve}$ & $T_{sum}$ & $its$ & $T_{prec}$ & $T_{solve}$ & $T_{sum}$\\
\hline
 1&   30& 5.18e+0& 1.19e+0& 6.37e+0          &       {\bf  30}   &{\bf 5.24e+0}& {\bf 1.18e+0}  & {\bf 6.42e+0}\\
 2&   12& 5.16e+0& 4.87e-1& 5.65e+0             &        14   &5.52e-1& 5.54e-1  & 1.11e+0\\
 3&    8& 5.16e+0& 3.40e-1& 5.50e+0              &       15   &5.49e-1& 6.01e-1  & 1.15e+0\\
 4&    5& 5.12e+0& 2.24e-1& 5.35e+0               &       13   &5.11e-1& 5.29e-1  & 1.04e+0\\
 5&    5& 5.14e+0& 2.27e-1& 5.37e+0               &        36   &5.52e-1& 1.37e+0  & 1.92e+0\\
 6&    8& 5.16e+0& 3.37e-1& 5.50e+0               &      48   &6.00e-1& 1.79e+0  & 2.39e+0\\
 7&   10& 5.15e+0& 4.15e-1& 5.57e+0            &       {\bf  10}   &{\bf 5.24e+0}& {\bf 4.17e-1}  & {\bf 5.66e+0}\\
 8&   12& 5.16e+0& 4.93e-1& 5.65e+0             &        15  &1.56e-1& 5.89e-1  & 7.45e-1\\
 9&   14& 5.13e+0& 5.61e-1& 5.70e+0              &       22   &2.76e-1& 8.39e-1  & 1.11e+0\\
10&   14& 5.18e+0& 5.58e-1& 5.74e+0              &      41   &4.82e-1& 1.54e+0  & 2.02e+0\\
11&   16& 5.14e+0& 6.33e-1& 5.78e+0             &       78   &4.90e-1& 2.90e+0  & 3.39e+0\\
12&   17& 5.17e+0& 6.68e-1& 5.83e+0             &      139   &4.66e-1& 5.09e+0  & 5.56e+0\\
13&   19& 5.14e+0& 7.40e-1& 5.88e+0              &     {\bf 19}  & {\bf 5.25e+0} & {\bf 7.44e-1}  & {\bf 5.99e+0}\\
14&   21& 5.15e+0& 8.11e-1& 5.97e+0              &      31& 1.95e-1& 1.16e+0  & 1.36e+0\\
15&   24& 5.15e+0& 9.24e-1& 6.08e+0               &     62   &4.68e-1& 2.32e+0  & 2.79e+0\\
16&   26& 5.13e+0& 1.39e+0& 6.51e+0              &     86   &4.72e-1& 3.17e+0  & 3.64e+0\\
17&   47& 5.27e+0& 1.76e+0& 7.03e+0               &    160  &4.61e-1& 5.83e+0  & 6.29e+0\\
\hline
          & 288& 8.77e+1& 1.18e+1& 9.95e+1         &     819        & 2.20e+1& 3.06e+1& 5.26e+1\\
        \hline 
\end{tabular}
\end{center}
}
\medskip
\caption{CVXQP3-D: details for $\pex$ and $\pup$ with $q=50$.\label{tCVXQP3-D}}
\end{table}

\begin{table}
{\small
\begin{center}
\begin{tabular}{|c|rccc|rrrr|}
\hline
 &
\multicolumn{4}{|c|}{$\pex$} &  
\multicolumn{4}{|c|}{$\pup$  ($q$=100)}\\
\cline{2-9}
$I\!P \; it$ & $its$ & $T_{prec}$ & $T_{solve}$ & $T_{sum}$ & $its$ & $T_{prec}$ & $T_{solve}$ & $T_{sum}$\\
\hline
1&     1& 2.11e-1& 2.95e-2& 2.41e-1           &      {\bf 1}& {\bf  2.09e-1}& {\bf 2.95e-2}& {\bf 2.39e-1}\\
 2&    2& 2.09e-1& 4.53e-2& 2.54e-1           &    5&    3.48e-2& 9.29e-2& 1.28e-1\\
 3&    2& 2.09e-1& 4.84e-2& 2.58e-1           &    9&     6.82e-2& 1.54e-1& 2.22e-1\\
 4&    3& 2.09e-1& 6.55e-2& 2.75e-1           & {\bf 3} & {\bf 2.05e-1} & {\bf 6.15e-2} & {\bf 2.66e-1}\\
 5&    3& 2.13e-1& 6.29e-2& 2.76e-1           &  18  & 9.65e-3& 2.77e-1& 2.86e-1\\
 6&    4& 2.09e-1& 8.21e-2& 2.91e-1           & {\bf   4} & {\bf 2.01e-1}& {\bf 7.30e-2}& {\bf 2.74e-1}\\
 7&    4& 2.13e-1& 7.84e-2& 2.92e-1           &   16 & 4.92e-2& 2.56e-1& 3.06e-1\\
 8&    4& 2.09e-1& 7.76e-2& 2.87e-1           & {\bf    4} & {\bf 2.05e-1}& {\bf 7.30e-2}& {\bf 2.78e-1}\\
 9&    4& 2.13e-1& 7.51e-2& 2.88e-1           &   19&   7.03e-2& 2.92e-1& 3.63e-1\\
10&    5& 2.09e-1& 9.20e-2& 3.01e-1          &  {\bf   5} & {\bf 2.01e-1}& {\bf 8.67e-2}& {\bf 2.87e-1}\\
11&    5& 2.17e-1& 8.93e-2& 3.06e-1          &   15 & 2.31e-2& 2.35e-1& 2.58e-1\\
12&    6& 2.13e-1& 1.09e-1& 3.23e-1          &   62& 6.90e-2& 9.12e-1& 9.81e-1\\
13&    5& 2.13e-1& 8.95e-2& 3.03e-1          & {\bf   5} & {\bf 2.09e-1}& {\bf 8.42e-2} & {\bf 2.93e-1}\\
14&    6& 2.13e-1& 1.07e-1& 3.20e-1          &    21 & 4.69e-2& 3.22e-1& 3.69e-1\\
15&    6& 2.13e-1& 1.07e-1& 3.20e-1          & {\bf    6} & {\bf 2.01e-1}& {\bf 1.01e-1}& {\bf 3.02e-1}\\
16&    6& 2.09e-1& 1.04e-1& 3.13e-1          &   22& 2.73e-2& 3.35e-1& 3.63e-1\\
\hline
   &  66 & 3.38e+0&1.27e+0 &4.65e+0   &            215& 1.83e+0& 3.39e+0& 5.21e+0 \\
\hline
\end{tabular}
\end{center}
}
\medskip
\caption{MOSARQP1: details for $\pex$ and $\pup$ with $q=100$.\label{tMOSARQP1}}
\end{table}

\begin{table}
{\small
\begin{center}
\begin{tabular}{|c|rccc|rrrr|}
\hline
 &
\multicolumn{4}{|c|}{$\pex$} &  
\multicolumn{4}{|c|}{$\pup$  ($q$=100)}\\
\cline{2-9}
$I\!P \; it$ & $its$ & $T_{prec}$ & $T_{solve}$ & $T_{sum}$ & $its$ & $T_{prec}$ & $T_{solve}$ & $T_{sum}$\\
\hline
1&    1& 1.95e+0& 5.99e-2& 2.01e+0           &  {\bf   1}& {\bf 1.90e+0}& {\bf 5.93e-2}& {\bf 1.96e+0}\\
 2&    1& 1.92e+0& 5.92e-2& 1.98e+0           &    4& 4.78e-2& 1.34e-1& 1.82e-1\\
 3&    1& 1.91e+0& 5.85e-2& 1.97e+0           &    8 & 8.98e-2& 2.46e-1& 3.36e-1 \\
 4&    2& 1.93e+0& 9.02e-2& 2.02e+0           &   15& 7.31e-1& 4.36e-1& 1.17e+0 \\
 5&    2& 1.91e+0& 9.01e-2& 2.00e+0           &   22& 7.31e-1& 6.31e-1& 1.36e+0\\
 6&    3& 1.91e+0& 1.18e-1& 2.03e+0           &   57 & 5.05e-1& 1.56e+0& 2.07e+0\\
 7&    3& 1.91e+0& 1.19e-1& 2.03e+0           &  {\bf  3} & {\bf 1.89e+0}& {\bf 1.13e-1}& {\bf 2.00e+0}\\
 8&    4& 1.90e+0& 1.51e-1& 2.05e+0           &   15& 5.10e-2& 4.39e-1& 4.90e-1\\
 9&    4& 1.90e+0& 1.48e-1& 2.05e+0           &   20 & 5.09e-2& 5.74e-1& 6.25e-1\\
10&    4& 1.91e+0& 1.47e-1& 2.06e+0          &   28 & 3.82e-1& 7.92e-1& 1.17e+0\\
11&    4& 1.90e+0& 1.50e-1& 2.05e+0          &   56 & 7.17e-1& 1.55e+0& 2.27e+0\\
12&    4& 1.91e+0& 1.46e-1& 2.05e+0          &   {\bf 4}& {\bf 1.89e+0}& {\bf 1.42e-1}& {\bf 2.03e+0}\\
13&    4& 1.91e+0& 1.47e-1& 2.06e+0          &    5 & 9.13e-2& 1.64e-1& 2.55e-1\\
14&    4& 1.93e+0& 1.48e-1& 2.08e+0          &    9& 1.03e-1& 2.69e-1& 3.72e-1\\
15&    3& 1.91e+0& 1.19e-1& 2.03e+0          &   17& 1.03e-1& 4.86e-1& 5.89e-1\\
16&    4& 1.92e+0& 1.46e-1& 2.06e+0          &   48 & 7.10e-1& 1.32e+0& 2.03e+0\\
17&    4& 1.94e+0& 1.45e-1& 2.08e+0          &   {\bf  6} & {\bf 1.88e+0}& {\bf 1.94e-1}& {\bf 2.07e+0}\\
18&    6& 1.92e+0& 2.03e-1& 2.13e+0          &   16 & 5.13e-2& 4.57e-1& 5.08e-1\\
19&    6& 1.91e+0& 2.07e-1& 2.11e+0          &   47 & 7.16e-1& 1.30e+0& 2.02e+0\\
20&    5& 1.91e+0& 1.75e-1& 2.08e+0          &   {\bf  6}& {\bf 1.89e+0}& {\bf 1.97e-1}& {\bf 2.08e+0}\\
21&    6& 1.91e+0& 2.06e-1& 2.11e+0          &   20 & 2.21e-1 & 5.69e-1 & 7.90e-1\\
22&    6& 1.91e+0& 2.03e-1& 2.12e+0          &   33 & 4.98e-1 & 9.19e-1 & 1.42e+0\\
23&    6& 1.93e+0& 2.03e-1& 2.13e+0         &        &               &               &                \\
24&    6& 1.91e+0& 2.02e-1& 2.12e+0         &        &               &               &                 \\
\hline
    &93& 4.60e+1& 3.44e+0& 4.94e+1       &         440& 1.52e+1& 1.26e+1& 2.78e+1\\
\hline
\end{tabular}
\end{center}
}
\medskip
\caption{MOSARQP1-D: details for $\pex$ and $\pup$ with $q=100$.\label{tMOSARQP1-D}}
\end{table}

\section{Conclusion} 
We have proposed a preconditioner updating  procedure for the solution of sequences of KKT systems
arising in IP methods for convex QP problems. The preconditioners built by this procedure belong to the class of inexact CPs and are obtained by updating a given seed CP. The updates are
performed through low-rank corrections of the Schur complement of the (1,1) block in the seed 
preconditioner and generate factorized preconditioners. The rule for identifying the low-rank
corrections  is based on new bounds on the eigenvalues of the preconditioned matrix.
The numerical experiments show that our updated preconditioners, combined with a suitable
preconditioner refreshing, can be rather successful.
More precisely, the higher the cost of the Schur complement factorization, the more advantageous 
the updating procedure becomes. Finally, we believe that the updating strategy proposed here paves the way to the definition
of preconditioner updating procedures for sequences of KKT systems where the Hessian and constraint matrices
change from one iteration to the next.

\vskip 10pt
\noindent
\textbf{Acknowledgments.}
We are indebted to Tim Davis for his valuable help in interfacing PRQP with CHOLMOD.
We also express our thanks to Miroslav T\.{u}ma for making available a Fortran 90 implementation
of SQMR used in preliminary numerical experiments. Finally, we wish to thank the anonymous
referees for their useful comments, which helped us to improve the quality of this work.


\clearpage

\begin{thebibliography}{99}


\bibitem{bsz}
V. Baryamureeba, T. Steihaug,  Y. Zhang,
{\em Properties of a Class of Preconditioners for Weighted Least Squares Problems},
Technical Report No. 170, Department of Informatics, University of Bergen, and
Technical Report No. TR99--16, Department of Computational and Applied Mathematics,
Rice University, Houston, 1999.

\bibitem{b_jota} S. Bellavia, {\em Inexact Interior-Point Method},
Journal of Optimization Theory and Applications, 96 (1998), pp.~109--121.

\bibitem{bbm}
S. Bellavia, D. Bertaccini, B. Morini, {\em
Nonsymmetric preconditioner updates in Newton-Krylov methods for nonlinear systems}, 
SIAM Journal on Scientific Computing, 33 (2011), pp.~2595--2619.

\bibitem{bddm_sisc}
S. Bellavia, V. De Simone, D. di Serafino,   B. Morini, 
{\em Efficient preconditioner updates for shifted linear systems},
SIAM Journal on Scientific Computing, 33 (2011), pp.~1785--1809.

\bibitem{bddm_sinum}
S. Bellavia, V. De Simone, D. di Serafino,   B. Morini, 
{\em A preconditioning framework for sequences of diagonally modified linear systems arising in optimization},
SIAM Journal on Numerical Analysis, 50 (2012), pp.~3280--3302.

\bibitem{bmp}
S. Bellavia, B. Morini, M. Porcelli,
{\em New updates of incomplete LU factorizations and applications to large nonlinear systems},
Optimization Methods and Software, 29 (2014), pp.~321--340.

\bibitem{bb}
M. Benzi, D. Bertaccini, {\em Approximate inverse
preconditioning for shifted linear systems}, BIT, 43 (2003), pp. 231--244.


\bibitem{bs}
M. Benzi, V. Simoncini,
{\em On the eigenvalues of a class of saddle point matrices},
Numerische Mathematik, 103 (2006), pp.~173--196.

\bibitem{b}
L. Bergamaschi, {\em Eigenvalue distribution of constraint-preconditioned symmetric saddle point
matrices}, Numerical Linear Algebra with Applications, 19 (2012), pp.~754--772.

\bibitem{bbmp} 
L. Bergamaschi, R. Bru, A. Martinez, M. Putti, 
{\em  Quasi-Newton preconditioners for the inexact Newton method},
Electronic Transactions on  Numerical Analysis, 23 (2006) pp.~76--87.

\bibitem{bgz} 
L. Bergamaschi, J. Gondzio,  G. Zilli,
{\em  Preconditioning Indefinite Systems in Interior Point Methods for Optimization}, 
Computational Optimization and Applications, 28 (2004), pp.~149--171. 

\bibitem{bgvz} 
L. Bergamaschi, J. Gondzio, M. Venturin, G. Zilli,
{\em Inexact Constraint Preconditioners  for Linear Systems Arising in Interior Point Methods},
Computational Optimization and Applications,  36 (2007), pp.~137--147.


\bibitem{cddd_coap1}
S. Cafieri, M. D'Apuzzo, V. De Simone, D. di Serafino, {\em On the iterative solution of KKT systems
in potential reduction software for large-scale quadratic problems}, Computational Optimization
and Applications, 38 (2007), pp.~27--45.

\bibitem{cddd_coap2}
S. Cafieri, M. D'Apuzzo, V. De Simone, D. di Serafino, {\em Stopping criteria for inner iterations in
inexact potential reduction methods: a computational study}, Computational Optimization
and Applications, 36 (2007), pp.~165--193.

\bibitem{cddd_simai}
S. Cafieri, M. D'Apuzzo, V. De Simone, D. di Serafino,
{\em On the Use of an Approximate Constraint Preconditioner in a Potential Reduction Algorithm for
Quadratic Programming}, in ``Applied and Industrial Mathematics in Italy II'', V. Cutello, G. Fotia and L. Puccio eds.,
Series on Advances in Mathematics for Applied Sciences, 75, World Scientific, 2007, pp. 220--230.

\bibitem{cdddt}
S. Cafieri, M. D'Apuzzo, V. De Simone, D. di Serafino,  G. Toraldo, {\em Convergence analysis of an
inexact potential reduction method for convex quadratic programming}, Journal of Optimization
Theory and Applications, 135 (2007), pp.~355--366.

\bibitem{ccs}
C. Calgaro, J.P. Chehab, Y. Saad, 
{\em Incremental incomplete ILU factorizations with applications},
Numerical Linear Algebra with Applications, 17 (2010), pp.~811--837.

\bibitem{cdg}
B. Carpentieri, I.S. Duff,  L. Giraud,
{\em A class of spectural two-level preconditioners},
SIAM Journal on Scientific Computing,  25 (2003), pp.~749--765.

\bibitem{ddd}
M. D'Apuzzo, V. De Simone, D. di Serafino, 
{\em On mutual impact of numerical linear algebra and 
large-scale optimization with focus on interior point methods}, 
Computational Optimization and Applications, 45 (2010), pp.~283--310.

\bibitem{ddd2}
M. D'Apuzzo, V. De Simone, D. di Serafino,
{\em Starting-Point Strategies for an Infeasible Potential Reduction Method},
Optimization Letters, 4 (2010), pp.~131--146.

\bibitem{dh}
T.A. Davis, W.W. Hager, {\em Dynamic supernodes in sparse Cholesky update/downdate and
triangular solves}, ACM Transactions on Mathematical Software, 35 (2009), article 27.

\bibitem{d}
H.S. Dollar, {\em Constraint-style preconditioners for regularized saddle-point problems},
SIAM Journal on Matrix Analysis and Applications, 29 (2007), pp.~672--684.

\bibitem{dgsw} 
H.S. Dollar, N.I.H. Gould, W.H.A. Schilders, A.J. Wathen,
{\em Implicit-factorization preconditioning and iterative solvers for regularized saddle-point
systems}, SIAM Journal on Matrix Analysis and Applications, 28 (2006), pp.~170--189.

\bibitem{dw} H.S. Dollar, A.J. Wathen,
{\em Approximate factorization constraint preconditioners for saddle-point matrices}
SIAM Journal on Scientific Computing, 27, (2006), pp.~1555--1572. 

\bibitem{tt1} J. Duintjer Tebbens, M. T\.{u}ma, {\em   Efficient Preconditioning of Sequences
of Nonsymmetric Linear Systems}, SIAM Journal on  Scientific Computing, 29 (2007),
pp.~1918--1941.

\bibitem{dr}
C. Durazzi, V. Ruggiero, 
{\em  Indefinitely preconditioned conjugate gradient method for large sparse
equality and inequality constrained quadratic problems}, Numerical Linear Algebra with Applications,
10 (2003), pp.~673--688.

\bibitem{fn2}
R.W. Freund, N.M. Nachtigal, 
{\em QMR: a quasi-minimal residual method for non-Hermitian linear systems},
Numerische Mathematik, 60 (1991), pp.~315--339.

\bibitem{fn}
R.W Freund, N.W Nachtigal, {\em Software for simplified Lanczos and QMR algorithms},
Applied Numerical Mathematics, 19 (1995), pp.~319--341.

\bibitem{fn_qmrpack}
R.W Freund, N.W Nachtigal, {\em QMRPACK: a package of QMR algorithms},
ACM Transactions on Mathematical Software, 22 (1996), pp. 46--77.

\bibitem{fr} 
G. Fasano, M. Roma, {\em Preconditioning Newton--Krylov Methods in Non-Convex Large Scale Optimization}, 
Computational Optimization and Applications, 56 (2013), pp.~253--290.

\bibitem{fgg}
A. Forsgren P.E. Gill, J.D. Griffin, {\em Iterative solution
of augmented systems arising in interior methods},
SIAM Journal on Optimization, 18 (2007), pp.~666--690.

\bibitem{ggm}
L. Giraud, S. Gratton, E. Martin,
{\em Incremental spectral preconditioners for sequences of linear systems},
Applied Numerical Mathematics, 57 (2007), pp.~1164--1180.

\bibitem{gv}
G.H. Golub, C.F. van Loan, {\em Matrix Computations}, third edition,
The Johns Hopkins University Press, 1996.

\bibitem{g}
J. Gondzio,
{\em Interior point methods 25 years later},
European Journal of Operational Research, 218 (2012), pp.~587-–601.


\bibitem{got2}
N.I.M. Gould, D. Orban, Ph.L. Toint,
{\em CUTEst: a Constrained and Unconstrained Testing Environment with safe threads},
Technical Report RAL-TR-2013-005, STFC Rutherford Appleton Laboratory, Chilton, Oxfordshire, UK,
2013.

\bibitem{gst}
S. Gratton, A. Sartenaer,  J. Tshimanga, {\em 
On a class of limited memory preconditioners for large scale linear systems with multiple right-hand sides},
SIAM Journal on Optimization, 21 (2011), pp.~912--935.

\bibitem{kgw}
C. Keller, N.I.M. Gould, A.J. Wathen, {\em Constraint preconditioning for indefinite linear systems},
SIAM Journal on Matrix Analysis and Applications, 21 (2000), pp.~1300–-1317.

\bibitem{lrt}
D. Loghin, D. Ruiz,  A. Touhami,
{\em Adaptive preconditioners for nonlinear systems of equations},
Journal of Computational and Applied Mathematics, 189 (2006), pp.~362--374. 

\bibitem{lv}
L. Luk\v{s}an, J. Vl\v{c}ek,
{\em Indefinitely Preconditioned Inexact Newton Method for
Large Sparse Equality Constrained Non-linear
Programming Problems}
Numerical Linear Algebra with Applications, 5 (1998), pp.~219--247.

\bibitem{m} 
G. Meurant, 
{\em On the incomplete Cholesky decomposition of a class of perturbed matrices},
SIAM J. Sci. Comput.,  23 (2001), pp.~419--429. 

\bibitem{mn}
J.L. Morales, J. Nocedal, {\em Automatic preconditioning by limited memory quasi-Newton updating},
SIAM Journal on Optimization, 10 (2000), pp.~1079--1096.

\bibitem{ps}
I. Perugia, V. Simoncini,
{\em Block-diagonal and indefinite symmetric preconditioners for mixed finite element formulations}, 
Numerical Linear Algebra with Applications, 7 (2000), pp.~585--616.


\bibitem{gmres}
Y. Saad, M.H. Schultz, {\em GMRES: a generalized minimal residual 
algorithm for solving nonsymmetric linear systems}, 
SIAM Journal on Scientific and Statistical Computing, 7 (1986), pp.~856--869.

\bibitem{ss}
D. Sesana, V. Simoncini, {\em Spectral analysis of inexact constraint preconditioning for symmetric
saddle point matrices}, 
Linear Algebra and its Applications, 438 (2013), pp.~2683--2700. 

\bibitem{wol}
W. Wang, D.P. O'Leary, {\em Adaptive use of iterative methods in predictor-corrector interior point methods
for linear programming}, Numerical Algorithms, 25 (2000), pp.~387--406.

\bibitem{w}
S.J. Wright, {\em Primal-Dual Interior-Point Methods}, SIAM, Philadelphia, 1997.

\end{thebibliography}
\end{document}